\documentclass[11pt]{article} 

 \usepackage{booktabs}       
 \usepackage{amsfonts}       
 \usepackage{nicefrac}       
 \usepackage{microtype}      

\usepackage{fncylab,enumerate,wrapfig,placeins}
\usepackage{bbm}
\usepackage{graphicx}
\usepackage{enumitem}  
\usepackage{amsmath,amsthm,amsfonts,amssymb}  
\usepackage{pifont}

\usepackage{geometry}

\pdfoutput=1  

\usepackage{hyperref}       
 \usepackage{url}            

\usepackage{fncylab,enumerate,wrapfig,placeins}

\usepackage{graphicx}  
\usepackage{relsize} 
\usepackage{accents}  

\usepackage{caption}
\usepackage{textgreek,upgreek,bm}

\usepackage{titlesec}

\def\ODEstateFB{\ODEstate^{\text{\tiny\sf FB}}}

\def\thetaopt{\theta^{\text{\tiny\sf opt}}}

\def\intI{\text{\rm\tiny I}}

\def\haUpupsilon{\widehat{\Upupsilon}}

\def\haf{\hat{f}}

\def\fq{f^\text{\rm\tiny $1$Q}}
\def\fqq{f^\text{\rm\tiny $2$Q}}

\graphicspath{{../figures/}}

\def\tilqsaprobe{\widetilde{\qsaprobe}}

\def\clD{\mathcal{D}}

\def\Df{\clD^f}

\newtheorem{theorem}{Theorem}[section]
\newtheorem{lemma}[theorem]{Lemma} 
\newtheorem{proposition}[theorem]{Proposition} 

\newtheorem{corollary}[theorem]{Corollary}

\newlength{\noteWidth}
\long\def\notes#1{\ifinner
	{\footnotesize #1}
	\else 
	\marginpar{\parbox[t]{\noteWidth}{\raggedright\tiny#1}}  
	\fi\typeout{#1}}

\def\notes#1{\typeout{check notes!!!}}   

\def\rd#1{{\color{red}#1}}

\usepackage{cleveref}

\Crefname{corollary}{Corollary}{Corollaries}
\Crefname{eqnarray}{eq.}{eqs.}
\Crefname{equation}{eq.}{eqs.}

\Crefname{figure}{Fig.}{Figs.}
\Crefname{tabular}{Tab.}{Tabs.}
\Crefname{table}{Tab.}{Tabs.}
\Crefname{proposition}{Prop.}{Propositions}
\Crefname{theorem}{Thm.}{Thms.}
\Crefname{definition}{Def.}{Defs.} 
\Crefname{section}{Section}{Sections}
\Crefname{lemma}{Lemma}{Lemmas}
\Crefname{assumption}{Assumption}{Assumptions}

\def\urls#1{{\footnotesize\url{#1}}}

\def\mindex#1{\index{#1}}

\def\ocp{*}   

\DeclareFontFamily{U}{mathx}{\hyphenchar\font45}
\DeclareFontShape{U}{mathx}{m}{n}{<-> mathx10}{}
\DeclareSymbolFont{mathx}{U}{mathx}{m}{n}
\DeclareMathAccent{\widebar}{0}{mathx}{"73}

\def\barUpupsilon{\widebar{\Upupsilon}}

\newcommand{\qsaprobe}{{\scalebox{1.1}{$\upxi$}}}  
\newcommand{\bfqsaprobe}{{\scalebox{1.1}{$\bm{\upxi}$}}}  
\def\Lip{L}  
 
\def\Obj{\Upgamma}  

\def\ODEstate{\Uptheta} 

\def\barODEstate{\widebar{\Uptheta}}

\def\ODEstatePR{\ODEstate^{\text{\tiny\sf PR}}}
\def\ODEstatePRm{\ODEstate^{\text{\tiny\sf PR$-$}}}

\def\odestate{\upvartheta}

\newcommand{\bbblot}{\raise1pt\hbox{\vrule height .4ex width .4ex depth .05ex}}

\long\def\defbox#1{\framebox[.9\hsize][c]{\parbox{.85\hsize}{%
\parindent=0pt
\baselineskip=12pt plus .1pt      
\parskip=6pt plus 1.5pt minus 1pt 
 #1}}}

\long\def\beginbox#1\endbox{\subsection*{}%
\hbox{\hspace{.05\hsize}\defbox{\medskip#1\bigskip}}%
\subsection*{}}

\def\endbox{}

 \def\archival#1{} 

\def\FRAC#1#2#3{\genfrac{}{}{}{#1}{#2}{#3}}

\def\ddt{{\mathchoice{\FRAC{1}{d}{dt}}%
{\FRAC{1}{d}{dt}}%
{\FRAC{3}{d}{dt}}%
{\FRAC{3}{d}{dt}}}}

\def\ddtp{{\mathchoice{\FRAC{1}{d^{\hbox to 2pt{\rm\tiny +\hss}}}{dt}}%
{\FRAC{1}{d^{\hbox to 2pt{\rm\tiny +\hss}}}{dt}}%
{\FRAC{3}{d^{\hbox to 2pt{\rm\tiny +\hss}}}{dt}}%
{\FRAC{3}{d^{\hbox to 2pt{\rm\tiny +\hss}}}{dt}}}}

\def\ddyp{{\mathchoice{\FRAC{1}{d^{\hbox to 2pt{\rm\tiny +\hss}}}{dy}}%
{\FRAC{1}{d^{\hbox to 2pt{\rm\tiny +\hss}}}{dy}}%
{\FRAC{3}{d^{\hbox to 2pt{\rm\tiny +\hss}}}{dy}}%
{\FRAC{3}{d^{\hbox to 2pt{\rm\tiny +\hss}}}{dy}}}}

\def\half{{\mathchoice{\FRAC{1}{1}{2}}%
{\FRAC{1}{1}{2}}%
{\FRAC{3}{1}{2}}%
{\FRAC{3}{1}{2}}}}

\def\darrow{\buildrel{\rm d}\over\longrightarrow}

\def\tr{{\rm tr\, }}

\def\limsup{\mathop{\rm lim{\,}sup}}

\def\bfmath#1{{\mathchoice{\mbox{\boldmath$#1$}}%
{\mbox{\boldmath$#1$}}%
{\mbox{\boldmath$\scriptstyle#1$}}%
{\mbox{\boldmath$\scriptscriptstyle#1$}}}}

\def\bfPhi{\bfmath{\Phi}}

\def\bfxi{\bfmath{\xi}}
 \def\bfzeta{\bfmath{\zeta}}

\def\bfma{\bfmath{a}}

\def\bfmY{\bfmath{Y}}

\def\bfmhhaY{\bfmath{\hhaY}} 
\def\bfmhhaY{\hbox to 0pt{$\widehat{\bfmY}$\hss}\widehat{\phantom{\raise 1.25pt\hbox{$\bfmY$}}}}

\def\haf{{\hat f}}

\def\hah{{\hat h}}

\def\haA{\widehat A}

\def\tilA{\tilde{A}}

\def\tilf{\tilde f}

\def\clN{{\cal N}}

\def\clW{{\cal W}}

\def\eqdef{\mathbin{:=}}

\def\Expect{{\sf E}}

\def\lgmath#1{{\mathchoice{\mbox{\large #1}}%
{\mbox{\large #1}}%
{\mbox{\tiny #1}}%
{\mbox{\tiny #1}}}}

\def\Zero{{\mathchoice{\lgmath{\sf 0}}%
{\mbox{\sf 0}}%
{\mbox{\tiny \sf 0}}%
{\mbox{\tiny \sf 0}}}}

 \def\epsy{\varepsilon}

\def\varble{\,\cdot\,}

\def\formtmp#1#2{{\vskip12pt\noindent\fboxsep=0pt\colorbox{#1}{\vbox{\vskip3pt\hbox to \textwidth{\hskip3pt\vbox{\raggedright\noindent\textbf{#2\vphantom{Qy}}}\hfill}\vspace*{3pt}}}\par\vskip2pt%
\noindent\kern0pt}}

\titleformat\subparagraph[runin]
                        {\normalfont\normalsize\bfseries}
                        {mypar}
                        {0pt}
                        {}{}
\titlespacing\subparagraph{0pt}
                       {.1ex minus 0.2ex}
                       {.75em}

\def\barf{{\widebar{f}}}

\def\barh{{\overline {h}}}

\def\barA{{\bar{A}}}

\def\barY{{\bar{Y}}}

\def\barSigma{\overline{\Sigma}}

{\end{list}}

\def\ass(#1:#2){(#1\ref{#1:#2})}

\def\ritem#1{
\item[{\sf \ass(\current_model:#1)}]
}

\newenvironment{recall-ass}[1]{%
\begin{description}
\def\current_model{#1}}{
\end{description}
}

\def\sq{\hbox{\rlap{$\sqcap$}$\sqcup$}}
\def\qed{\ifmmode\sq\else{\unskip\nobreak\hfil
\penalty50\hskip1em\null\nobreak\hfil\sq
\parfillskip=0pt\finalhyphendemerits=0\endgraf}\fi}

\newcommand{\blot}{\vrule height 1.1ex width .9ex depth -.1ex }
\def\qedb{\ifmmode\blot\else{\vspace{-.2cm}\unskip\nobreak\hfil
\penalty50\hskip1em\null\nobreak\hfil\blot
\parfillskip=0pt\finalhyphendemerits=0\endgraf}\fi}

\newtheoremstyle{example}{15pt}{20pt}%
     {}
     {}
     {\bfseries}
     {}
     {1pt}
     {\thmname{#1}\thmnumber{ #2.}~\thmnote{\textit{\textbf{#3}}}%
     \\[.15cm]\unskip\nobreak}

\newcounter{rmnum}
\newenvironment{romannum}{\begin{list}{{\upshape (\roman{rmnum})}}{\usecounter{rmnum}
\setlength{\leftmargin}{18pt}
\setlength{\rightmargin}{8pt}
\setlength{\itemindent}{2pt}
}}{\end{list}}

\newcounter{anum}
\newenvironment{alphanum}{\begin{list}{{\upshape (\alph{anum})}}{\usecounter{anum}
\setlength{\leftmargin}{18pt}
\setlength{\rightmargin}{8pt}
\setlength{\itemindent}{2pt}
}}{\end{list}}

\newcommand{\field}[1]{\mathbb{#1}}

\def\Re{\field{R}} 

\def\intgr{\field{Z}}
\def\nat{\field{Z}_+}

\def\Co{\field{C}}

\def\Expect{{\sf E}}

\def\transpose{{\intercal}}

\def\epsy{\varepsilon}
\def\varble{\,\cdot\,}

\def\haY{\widehat{Y}}

\def\hhaY{\hbox to 0pt{$\haY$\hss}\widehat{\phantom{\raise 1.25pt\hbox{Y}}}}

\def\haA{\widehat A}

\def\haS{{\widehat S}}

\def\haY{\widehat Y}

\def\bfPhi{\bfmath{\Phi}}

\newlength{\dhatheight}

    \newcommand{\doublehat}[1]{%
    \settoheight{\dhatheight}{\ensuremath{\hat{#1}}}%
    \addtolength{\dhatheight}{-0.3ex}%
    \hat{\vphantom{\rule{1pt}{\dhatheight}}%
    \smash{\hat{#1}}}}
    
\def\hahaf{\doublehat{f}}
\def\hahah{\doublehat{h}}

\def\tilUpupsilon{\widetilde{\Upupsilon}}
\def\barUpupsilon{\widebar{\Upupsilon}}

\def\tilXi{\widetilde{\Xi}}

\def\prstate{\Upomega}

\def\qsaDyn{\text{H}}

\def\Sigmaqsa{\Sigma_{\text{\tiny$\qsaprobe$}}}

\def\tilh{\tilde{h}}

\makeatletter
\newcommand\gobblepars{%
    \@ifnextchar\par%
 {\expandafter\gobblepars\@gobble}%
{}}
\makeatother

\def\wham#1{\smallbreak\pagebreak[3]%
\noindent\textbf{#1}\ \ \gobblepars}

 \def\witem{\wham{\small$\triangle$}}

\title{Extremely Fast Convergence Rates for 
\\
Extremum Seeking Control with Polyak-Ruppert Averaging}

\author{%
	Caio Kalil Lauand \\
	University of Florida\\
	\texttt{caio.kalillauand@ufl.edu} \\[.5em]
	Sean Meyn\\
	University of Florida\\
	\texttt{meyn@ece.ufl.edu} \\
}

\begin{document}

\clearpage
	
	\setcounter{page}{1}

	\maketitle

	\begin{abstract}
Stochastic approximation is a foundation for many algorithms found in machine learning and optimization.   It is in general slow to converge:  the mean square error vanishes as $O(n^{-1})$.    A deterministic counterpart known as quasi-stochastic approximation is a viable alternative in many applications, including gradient-free optimization and reinforcement learning.   It was assumed in prior research that the optimal achievable convergence rate is $O(n^{-2})$.    It is shown in this paper that through design it is possible to obtain far faster convergence,   of order $O(n^{-4+\delta})$,  with $\delta>0$  arbitrary.   
Two techniques are introduced for the first time to achieve this rate of convergence.   The theory is also specialized within the context of gradient-free optimization, and tested on standard benchmarks.  The main results are based on a combination 
of novel application of 
results from number theory and techniques adapted from stochastic approximation theory.

\smallskip
\noindent
Financial support from ARO award W911NF2010055
and National Science Foundation award EPCN 1935389
is gratefully acknowledged

\noindent
AMS-MSC Codes: 62L20, 34C29
%

	\end{abstract}

\clearpage	
		\tableofcontents

\clearpage

	\section{Introduction}
	\label{s:Intro}
	Stochastic approximation (SA) was   introduced in the seminal work of Robbins and Monro \cite{robmon51a}. The goal is to solve the root finding problem $\barf(\theta^\ocp) =0$, in which $\barf\colon\Re^d\to\Re^d$ is of the form 
\begin{equation}
	\barf(\theta) \eqdef \Expect[f(\theta,\zeta)]
	\label{e:barf}
\end{equation}
where $\zeta$ is a random vector taking values in $\Re^m$.  The basic algorithm is expressed as the $d$-dimensional recursion,
\begin{equation}
	\theta_{n+1} = \theta_{n}+ \alpha_{n+1} f(\theta_n,\zeta_{n+1}) \, , \quad n\ge 0,
	\label{e:SA_recur}
\end{equation}
in which $\{\alpha_n\}$ is the   step-size sequence, and $\zeta_{n+1} \darrow \zeta$ as $n \to \infty$  (convergence  in distribution).  
SA theory has attracted a great deal of attention over the past twenty years, motivated in large part by applications to reinforcement learning and optimization   \cite{tsi94a,kontsi04,bacmou11,karbha18,bor20a}. 


Convergence theory is couched in the \textit{ODE Method} in which trajectories of \eqref{e:SA_recur} are compared to solutions of the ODE $\dot \odestate = \barf(\odestate)$ (the \textit{mean flow}).   The major assumption required to ensure convergence to $\theta^\ocp$, for each initial condition  $\theta_0 \in \Re^d$,  is that  the mean flow is globally asymptotically stable---see  \cite{bor20a}  for minimal assumptions on the ``noise sequence'' $\bfzeta$.   
Establishing sharp rates of convergence is a far greater challenge.   There is however a rich theory available to achieve the optimal rate of convergence for the mean-square error (MSE), which is in general   $\Expect[ \| \theta_n- \theta^\ocp\|^2] = O(n^{-1})$.

There are many applications for which the designer of the algorithm also designs the noise.   Notable examples include the introduction of exploration in reinforcement learning or gradient-free optimization.  This motivates the use of \textit{quasi-stochastic approximation} (QSA) in which the sequence $\bfzeta$ is deterministic (e.g. mixtures of sinusoids or pseudo-random numbers).  The idea was introduced in \cite{lappagsab90,larpag12}, but has a much longer history in the context of gradient-free optimization---see \cite{EShistory2010,liukrs12} for a survey of \textit{extremum seeking control} (ESC).

Theory supporting rates of convergence of nonlinear QSA appeared only recently \cite{chedevbermey21,CSRL}.   Analysis and algorithms are posed in continuous time to simplify analysis.   This setting is also motivated by recent success stories justifying algorithm design in continuous time, followed by a careful translation to obtain a discrete time algorithm.   See for example theory surrounding acceleration methods of  Polyak and Nesterov   \cite{attgoured00,suboycan14,kovstu21}.

The notation adopted in \cite[Chap.~4]{CSRL}  will be used here:   the \textit{QSA ODE} is defined as 
\begin{equation}
	\ddt\ODEstate_t = a_t f(\ODEstate_t,\qsaprobe_t) 
	\label{e:QSAgen}
\end{equation}
The deterministic continuous time process $\bfqsaprobe$ will be called the \textit{probing signal}, and
plays the role of $\bfzeta$ in SA;    $\bfma$ is called the gain process.   The motivation for QSA is two-fold:
\begin{romannum}
	\item
	It will be seen that the rate of convergence is far faster than SA,  subject to careful choice of algorithm architecture.
	
	\item 
	In on-line applications the introduction of independent noise may not be realizable, or may impose unnecessary stress on equipment.  In QSA design the components of the probing signal might be chosen to be  sinusoidal signals of appropriate frequency and magnitude to ensure learning takes place without stress on physical devices. 
\end{romannum}

\textit{What is the optimal rate of convergence for QSA?}     Consider the most basic one-dimensional problem in which $ f(\ODEstate_t,\qsaprobe_t)   =  -  \ODEstate_t  +  \qsaprobe_t$,  with $\bfqsaprobe$ a zero-mean signal.    The special case $a_t = 1/(1+t)$ results in an approximate average:
\begin{equation}
	\ODEstate_T =  \frac{1}{1+T}  \ODEstate_0  +     \frac{1}{1+T} \int_0^T  \qsaprobe_t\,  dt
	\label{e:QSAfallacy}
\end{equation}
If for example $\qsaprobe_t  =   \sin(\omega t)$ then the right hand side converges at rate $O(T^{-1})$,  which translates to $O(T^{-2})$ for the ``MSE''; this is far faster than the rate $O(T^{-1})$ expected for SA.    

\begin{wrapfigure}[9]{r}{0.4\textwidth}
	
	\vspace{-.75em}
	\includegraphics[width=0.95\hsize]{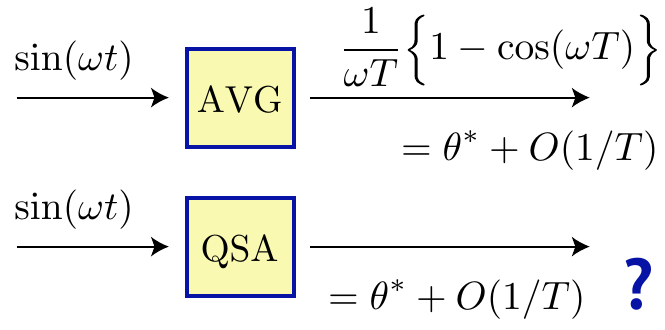}
	\caption{What is the optimal convergence rate for QSA?}\label{fig:diagram}
\end{wrapfigure}

The special case of pure averaging is illustrated at the top in \Cref{fig:diagram}.   The bold question mark in the figure refers to a question regarding a natural extension of the linear example in \eqref{e:QSAfallacy}:   \textit{can we obtain the same rate of convergence for general non-linear QSA?}  This question is posed and answered in the affirmative in \cite[\S~4.9]{CSRL}, achieved  through the averaging technique of Polyak and Ruppert,  but only under conditions on the QSA ODE that could not be verified a~priori. The current paper not only provides ways to ensure such conditions are always met, but also answers this question in a much more optimistic manner: the presumption that MSE rates of order $O(T^{-2})$ are optimal for general QSA is \textit{fallacy} since for any $\delta>0$, rates of order  $\| \ODEstate_T - \theta^\ocp\|^2 = O(T^{-4+\delta})$ can be achieved through design.

The astonishingly fast MSE rates obtained for general QSA are a result of the \textit{perturbative mean flow} (p-mean flow) representation for the QSA ODE. This representation expresses \eqref{e:QSAgen} in terms of the average vector field $\barf$,
\begin{align}
	\ddt \ODEstate_t    = 
	a_t[ \barf (\ODEstate_t) & 
	-   a_t 	\barUpupsilon_t   
	+ \clW_t] \,,  \qquad \clW_t =   \sum_{i =0}^2 a^{2-i}_t \frac{d^i}{dt^i} \clW^i_t
	\label{e:BigGlobalODEvanishing}
\end{align}
where $\{\clW_t^i : i=0,1,2\}$ are smooth functions of  a larger state process,   and $\barUpupsilon_t = \barUpupsilon(\ODEstate_t)$ with  $\barUpupsilon$ also smooth (see \Cref{t:P-meanflow}).

This representation is valuable only after boundedness of $ \{\ODEstate_t\}$ is established, for which sufficient conditions are provided in    \cite[Prop. 4.33 and 4.34]{CSRL}  (these conditions form part of Assumption (QSA3) in the supplementary material).    
A linearization of \eqref{e:BigGlobalODEvanishing} around $\theta^*$ gives
\begin{equation}
	\ddt \ODEstate_t    =  
	a_t  [ A^*(\ODEstate_t-\theta^*)
	- a_t \barUpupsilon^*
	+\clW_t + O(\| \ODEstate_t-\theta^* \|^2)] 
	\label{e:BigGlobalODElin}
\end{equation}
where $A^* = \partial \barf(\theta^*)$ and $\barUpupsilon^* = \barUpupsilon(\theta^*)$.    The upper bound $\|\ODEstate_t - \theta^*\| = O(a_t)$ (previously obtained in \cite{chedevbermey21,CSRL}) follows easily from \eqref{e:BigGlobalODEvanishing}.  

\notes{CL: are we coining a name to $\barUpupsilon^*$? Noise multiplication/multiplicative index ,  }

\begin{figure}[h]
	\includegraphics[width=1\hsize]{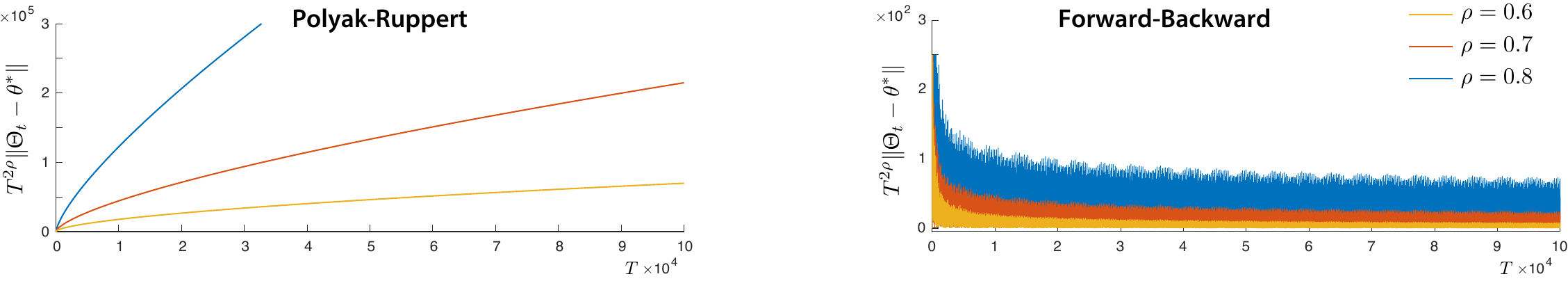}
	\caption{Rates of convergence for PR averaging and forward-backward algorithms.}
	\label{fig:comp_rho_ROC}
\end{figure} 

Given estimates $\{\ODEstate_t : t\ge 0\}$,   Polyak-Ruppert (PR) averaging defines the new estimates as follows,
\begin{equation}
	\ODEstatePR_T  = \frac{1}{T-T_0} \int_{T_0}^T  \ODEstate_t \,  dt \,, \quad T> T_0\,,  
	\label{e:PR_def}
\end{equation}
where the interval $[0,T_0]$ is known as the burn-in period.   The time-average of $\{\clW_t\}$ is of order $O(a_T^2)$, implying the rate $\| \ODEstatePR_T - \theta^\ocp\|^2 = O(T^{-4+\delta})$   if and only if $\barUpupsilon^* = 0$.

This brings us to the main contributions:
\begin{romannum}
	\item   The p-mean flow representation \eqref{e:BigGlobalODEvanishing}   is introduced for the first time in \Cref{t:P-meanflow}.
	
	\item   Convergence rates arbitrarily close to $O(T^{-2})$ can be obtained (so that the MSE is arbitrarily close to $O(T^{-4})$),  
	which is far faster than the bound  $O(T^{-1})$ assumed in prior research.
	
	\item The near quartic rates for MSE are established for PR averaging in \Cref{t:Couple_main}, subject to the assumption that $\barUpupsilon^*=0$.    
	Sufficient conditions on the probing signal  $\bfqsaprobe$ are provided in \Cref{t:P-meanflow} to ensure that  $\barUpupsilon^*=0$.

	\item An alternative to PR averaging is introduced:    \textit{forward-backward filtering}     achieves near quartic rates without the special conditions imposed in  \Cref{t:Couple_main}    (see \Cref{t:FB_ROC}).
	
\end{romannum}

The theory is refined in the context of gradient free-optimization,   and the general theory is illustrated through numerical examples in this setting.

\begin{wrapfigure}[14]{r}{0.375\textwidth}
	
	\vspace{-.95em}

	\includegraphics[width=0.95\hsize]{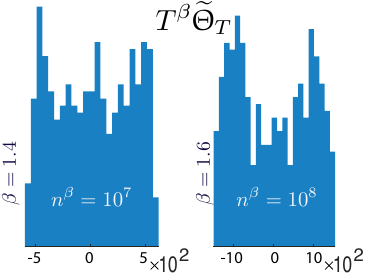}
	\caption{Quasi Monte Carlo using QSA with  Polyak-Ruppert averaging.  \\
		Histograms for  $\protect n=10\times T =10^5 $.}
	\label{fig:PRQSA}
\end{wrapfigure}

While theory is developed in continuous time,  simulation studies using   Euler approximations are consistent with theory.   The plots shown in \Cref{fig:comp_rho_ROC} are based on a two-dimensional example, whose details can be found in \Cref{s:QSA_intro}.  
The change of notation is also explained there:  $(2\rho)^2 =4-\delta$,  where $\rho$ is a design parameter subject to $\rho<1$.  
The forward-backward algorithm is a new algorithm introduced in this paper, that achieves the $ O(t^{-4+\delta})$ convergence rate;   we find in these experiments that the same rates hold for the Euler approximation. 
Polyak-Ruppert averaging \eqref{e:PR_def}  cannot obtain this fast rate of convergence,  because the frequencies used in this experiment violate a critical assumption in \Cref{t:Couple_main}.

\Cref{fig:PRQSA} illustrates how the techniques introduced in this paper specialize to quasi-Monte Carlo (QMC), based on two instances of the algorithm: one designed using  $\beta = \sqrt{4-\delta}=1.4$, and the other with $\beta=1.6$.  Details are found in the supplementary material.

\subsection{Preliminaries}
\label{s:Preliminaries}

A short literature survey and notation glossary are provided here. 

\wham{Literature review:}    Polyak-Ruppert (PR) averaging was introduced in \cite{rup88,pol90,poljud92}, and is now a standard workhorse in machine learning  \cite{frapol20,bacmou11,moujunwaibarjor20,CSRL}.    

Quasi-Monte Carlo (QMC) remains an active area of research for applications to estimation and optimization \cite{owegly18}.  
One aspect of the theory concerns techniques to construct a sequence $\{\zeta_i  :  i \in \nat  \} \subset  [0,1]^K$   so that sample path averages are convergent,  $n^{-1} \sum_{i=1}^n g(\zeta_i) \to   \int g(x)\, \mu(dx)$ as $n\to\infty$ for functions $g\colon\Re^K\to\Re$, with $\mu$ uniform;   theory predicts a convergence rate of $O(\log(n)^K /n)$  (see e.g.\  \cite[Section 9.3]{asmgly07}).

It is shown in \Cref{t:BakerCor} of the present paper that the convergence rate $O(1/n)$ holds for QMC,  subject to a smoothness condition on the function,  and careful design of the probing signal. \textit{The main results  of this paper hinge on these new bounds for QMC, and further refinements}.  These  bounds are based in part on refinements of Baker's Theorem \cite{mat00,bug18}.

QSA was first applied to finance in \cite{lappagsab90,larpag12}.
Stability and convergence theory for QSA is recent   \cite{chedevbermey21,CSRL},  following preliminary results in  \cite{berchecoldalmehmey19b};  this work was motivated by applications to extremum seeking control,  and successful application of QSA to Q-learning in \cite{mehmey09a}.

The linear QSA ODE is treated in  \cite{shimey11}, where the first bounds on the rate of convergence of order $O(a_t)$ were obtained.    This bound was extended to general nonlinear QSA in   \cite{chedevbermey21,CSRL},  and it was also shown that the use of PR averaging results in a convergence rate of   $O(T^{-1})$ subject to (up to now) unverifiable conditions.   

QSA with fixed step-size is the topic of the contemporaneous work  \cite{laumey22b,laumey22d}.   An analog of the p-mean flow representation \eqref{e:BigGlobalODEvanishing} is obtained, along with both steady-state and   transient error bounds

\begin{subequations}
	
	Gradient-free optimization (GFO) concerns minimization of   an objective function $\Obj: \Re^d \to \Re$ based solely on measurements of $\Obj(\theta)$ for selected values of $\theta\in\Re^d$.    
	The algorithms of Keifer and Wolfowitz   are early examples  \cite{kiewol52,bhaprapra13,bhabor03,gos15,asmgly07}.     The QSA ODEs for GFO considered in this paper are inspired by the   simultaneous perturbation stochastic approximation (SPSA) algorithms of   Spall  \cite{spa87, spa92, spa97}.
	Two examples are  of the form \eqref{e:SA_recur}, with
	\begin{align}
		\text{\sf 1SPSA:} \qquad 
		f^{1\text{SPSA}}(\theta_n,\zeta_{n+1}) &= -\frac{1}{\epsy} 
		\zeta_{n+1} \Obj(\theta_n + \epsy \zeta_{n+1})
		\\
		\text{\sf 2SPSA:} \qquad 
		f^{2\text{SPSA}}(\theta_n,\zeta_{n+1}) &=-
		\frac{1}{2\epsy} \bigl( \zeta_{n+1}[\Obj(\theta_n + \epsy \zeta_{n+1}) -\Obj(\theta_n - \epsy \zeta_{n+1})]  \bigr)
	\end{align}
	where $\epsy>0$ and $\{\zeta_n\}$ is a zero-mean i.i.d. sequence.
	
	\label{e:SPSA_Gen}
\end{subequations}

These algorithms are biased in general unless the probing gain $\epsy$ is a vanishing function of $n$.   There is substantial research in this setting:   the best possible convergence rate for the mean square error is  $O(n^{-\beta} )$ with  $\beta =  (p-1)/( 2p) $, provided the  objective function is $p$-fold differentiable at $\theta^\ocp$ \cite{poltsy90}.   Upper bounds appeared earlier in \cite{fab68b}.
See \cite{dipren97,dip03,spa12,pasgho13,larmenwil19} for more recent history.

Deterministic versions of SPSA were analyzed in \cite{bhafumarwan03} without sharp rates of convergence.   The present paper follows the approach of  \cite[\S~4.9]{CSRL}, in which the applications to gradient-free optimization result in algorithms that resemble simple versions of the ESC algorithms surveyed in \cite{EShistory2010,liukrs12}  (see \cite{laumey22d} for more on the relationship between ESC and QSA).

\wham{Notation:}   

We restrict to probing signals that are nonlinear functions of sinusoids, of the form
\begin{equation}
	\qsaprobe _t  = G_0(\qsaprobe^0_t) \quad 
	\text{with} 
	\quad  
	\qsaprobe^0_t = [\cos(2\pi[\omega_1 t + \phi_1 ] ), \cdots , \cos(2\pi[\omega_K t + \phi_K ] ) ]^\transpose
	\label{e:qSGD_probe0}
\end{equation}
for which 
$G_0: \Re^K \to \Re^m$ is analytic on $\Co^K$, $\{ \phi_i \}$ are arbitrary, and assumptions on the 
distinct 
frequencies $\{\omega_i\}$ will be imposed in the main results. When $G_0$ is linear we obtain the mixture of sinusoids
\begin{equation}   
	\qsaprobe _t  = \sum_{i=1}^{K_\bullet} v^i \cos(2 \pi [\omega_i t + \phi_i])\, , 
	\quad 
	\text{ for vectors $\{v^i\} \subset \Re^m$}
	\label{e:qSGD_probe}
\end{equation}
Denote $\Phi_t^i =   \exp(2\pi j[\omega_i t + \phi_i])$ for each $i$ and $t$,   which defines the $K$-dimensional vector-valued
$\Phi_t$     function of time evolving on $\prstate\eqdef \{ z \in\Co^K :  |z_i| = 1\,, \ 1\le i\le K\}$.      Writing  $ G(z) \eqdef G_0(  (z+ z^{-1})/2) $ for non-zero  $ z\in \Co^K$, we obtain $\qsaprobe_t   = G(\Phi_t ) $.

For two real-valued functions of time $\{\upzeta^i_t : t\ge 0\,, \ i=1,2 \}$ we denote
\begin{align}
	\langle \upzeta^1, \upzeta^2 \rangle   \eqdef   \lim_{T \to \infty}\frac{1}{T} \int_0^T   \upzeta^1_t \upzeta^2_t  \,  dt   \,,	
	\qquad
	\langle \upzeta^1 \rangle  \eqdef   \lim_{T \to \infty}\frac{1}{T}\int_0^T  \upzeta^1_t \,  dt \,,
	\qquad
	\tilde{\upzeta}^i_t = \upzeta^i_t-\langle \upzeta^i \rangle
	\label{e:spath-notation}
\end{align}
The notation is extended to vector or matrix valued functions.  In particular,  $ \langle \qsaprobe \rangle =\Zero\in\Re^m$,
and  $\Sigmaqsa \eqdef \langle\qsaprobe\qsaprobe^\transpose \rangle $ is an $m \times m$  matrix.

The functions of time used in \eqref{e:spath-notation}
are frequently obtained as functions of $(\ODEstate_t,\qsaprobe_t)$  [or more generally, $(\ODEstate_t,\Phi_t)$].
If $h\colon\Re^d\times\Re^m \to \Re$ we often write  $h_t \equiv h(\ODEstate_t,\qsaprobe_t)$ and $h^*_t \equiv h(\theta^*,\qsaprobe_t) $ to save space, and use the alternative notation  $\barh \eqdef \langle h \rangle$;   
identical notation is used for functions of $(\ODEstate_t,\Phi_t)$.

Let $\hah$ denote a solution of \textit{Poisson's equation} \eqref{e:Poissoneqdef} with forcing function $h\colon\Re^m\to\Re$,
which satisfies the defining equations, 
\begin{equation}
	\begin{aligned}
		\hah(\Phi_T)= -\int^T_0 \tilh(\qsaprobe_t) \,  dt + \hah(\Phi_0)  \,,\quad T\ge 0\,,\  \Phi_0\in\prstate
	\end{aligned} 
	\label{e:Poissoneqdef}
\end{equation}
It is assumed throughout that the solution is normalized so that its mean is zero.
\Cref{t:BakerCor} in the supplementary material provides conditions ensuring a smooth solution.

For a continuously differentiable ($C^1$) function $g\colon \Re^d \times \prstate \to \Re$, the directional derivative in the direction of the QSA vector field   $f$ is denoted
\begin{equation}
	[\clD^f g](\theta,z) = \partial_\theta g\,  (\theta,z )   \cdot f (\theta, G(z)) \, , 
	\quad 
	(\theta, z) \in \Re^d \times \prstate  
	\label{e:direc_derivat}
\end{equation}

For a scalar-valued function $\delta_t>0$, we use the notation $O(\delta_t)$ to denote a 
function of $t$;  the notation indicates that there is a constant $B$ such that $ \| O(\delta_t) \|  \leq B \delta_t$ for all $t$.   That is, we are only asserting an \textit{upper bound}.   For example, if $\delta_t =1/(1+t)$ we are not claiming that  $\| O(\delta_t) \| \leq B \delta_t^2$  is not possible.

\wham{Organization:}   \Cref{s:PRF}   summarizes  the main contributions of the paper, and
\Cref{s:GFO} provides illustrations of the theory in application to GFO; in particular, convergence rate bounds are validated for GFO algorithms based on QSA.  Conclusions and directions for future research are summarized in  \Cref{s:conc}. Technical proofs and additional numerical results can be found in the supplemental material.

\section{Towards Quartic Rates} 
\label{s:PRF}

\subsection{Quasi-Stochastic Approximation}
\label{s:QSA_intro}

We begin with foundations regarding the QSA ODE  \eqref{e:QSAgen} from \cite[Ch. 4.5, 4.9]{CSRL}. 
The \textit{mean vector field}  is defined by the sample path average,
\begin{align}
	\barf(\theta) &\eqdef \lim_{T \to \infty} \frac{1}{T} \int^T_0 f(\theta,\qsaprobe_t)\, dt \,, \quad \theta\in\Re^d
	\label{e:ergodicA1}
\end{align}
Theory is based on comparison with the ODE,
\begin{equation}
	\ddt\barODEstate_t = a_t \barf(\barODEstate_t)\, , 
	\quad 
	\text{for $t\geq t_0$ and $\barODEstate_{t_0} = \ODEstate_{t_0}$.}
	\label{e:barQSA}
\end{equation}
The time $t_0$ is fixed, but chosen suitably large in analysis of convergence rates.

The full list of assumptions  (QSA0)--(QSA5) may be found in the supplementary material.  We settle for a brief summary here: (QSA0) specifies conditions on the frequencies in \eqref{e:qSGD_probe0}.  The assumptions allow the following special case:   
fix $\omega_1>0$ and an increasing sequence of positive integers $\{n_i\}$, and choose the probing signal of the form 	\eqref{e:qSGD_probe0}  with 
\begin{equation}
	\omega_i = n_i \omega_1     \ \ \text{and} \ \  \phi_i = n_i \phi_1  \,, \quad 2 \leq i \leq K_\bullet  
	\label{e:specialcase_LogFreqSpan}
\end{equation}
See \Cref{t:OneFrequency} for explanation.

(QSA1) concerns the gain process; in the body of the paper we take $a_t = a_0(1+t)^{-\rho}$ for $\rho \in (1/2,1]$ and a constant $a_0>0$.  (QSA2) imposes global Lipschitz bounds on $f$ and $\barf$. 
(QSA3) imposes global asymptotic stability of the mean flow  $\ddt \odestate_t  = \barf  ( \odestate_t  )$ and minor additional assumptions.
(QSA4):   $\barf$ is $C^1$,  $\barA(\theta) =\partial_\theta \barf\, (\theta)$ is bounded and Lipschitz continuous,  and  $A^\ocp = A (\theta^\ocp)$ is Hurwitz. 
(QSA5):   existence of functions $\haf$,  and $\haUpupsilon$ that are Lipschitz continuous in $\theta$ solving for each $ 0\le t_0\le t_1$,
\begin{equation}
	\begin{aligned} 
		\haf( \theta, \Phi_{t_0} ) &= 
		\int_{t_0} ^{t_1}  \tilf(\theta,\qsaprobe_t)    \, dt    + \haf( \theta, \Phi_{t_1} ) \, , 
		\qquad  
		\tilf(\theta,\qsaprobe_t) = f(\theta,\qsaprobe_t) -\barf(\theta) 
		\\	
		\hahaf( \theta, \Phi_{t_0} ) &=  
		\int_{t_0} ^{t_1}   \haf(\theta,\qsaprobe_t)   \, dt    + \hahaf( \theta, \Phi_{t_1} ) 
		\\	
		\haUpupsilon( \theta, \Phi_{t_0} ) &= 
		\int_{t_0} ^{t_1}   \tilUpupsilon(\theta,\Phi_t)   \, dt    + \haUpupsilon( \theta, \Phi_{t_1} )\, ,
		\qquad 
		\tilUpupsilon(\theta,\Phi_t) = \Upupsilon(\theta,\Phi_t) -\barUpupsilon(\theta)
	\end{aligned} 
	\label{e:hatDefs}
\end{equation}
where for $\haA(\theta,\qsaprobe) \eqdef \partial_\theta \haf \, (\theta,\qsaprobe)$,
\begin{equation}
	\begin{aligned} 
		\Upupsilon(\theta,\Phi_t) &= -\haA\, (\theta,\Phi_t) f(\theta,\qsaprobe_t) \, ,\qquad \barUpupsilon(\theta) = \lim_{T \to \infty} \frac{1}{T} \int^T_0 \Upupsilon(\theta,\Phi_t)\, dt
	\end{aligned} 
	\label{e:Upupsilon}
\end{equation}

If $A(\theta,\qsaprobe) = \barA(\theta)$ for all $\theta$ and $\qsaprobe$ then   $\barUpupsilon(\theta^*) =0$.  This is a very restrictive special case, though it does hold for QMC for which $A(\theta,\qsaprobe) = -I$. In broad generality, we can always ensure $\barUpupsilon(\theta^*) =0$ by designing an algorithm with $\bfxi$ satisfying (QSA0) as shown in part (ii) of \Cref{t:P-meanflow}.

The solutions to Poisson's equation in \eqref{e:hatDefs} are used   to define the terms in  \eqref{e:BigGlobalODEvanishing}:

\begin{subequations}

	\begin{theorem}
		\label[theorem]{t:P-meanflow}
		Suppose that (QSA1) holds and $a_t = (1+t)^{-\rho}$, with $\rho\in(0,1)$.
		\begin{romannum}
			\item Under (QSA5), the p-mean flow representation \eqref{e:BigGlobalODEvanishing} holds with
			\begin{align} 
				&\clW^0(\ODEstate_t,\Phi_t) = -[D^f \haUpupsilon](\ODEstate_t,\Phi_t) 
				+ \frac{r_t}{a_t} [D^f \hahaf](\ODEstate_t,\Phi_t) 
				\\
				&\clW^1(\ODEstate_t,\Phi_t) = -[D^f \hahaf](\ODEstate_t,\Phi_t)+\haUpupsilon(\ODEstate_t,\Phi_t)
				\\
				&\clW^2(\ODEstate_t,\Phi_t) = \hahaf(\ODEstate_t,\Phi_t)\,, 
				\qquad 
				\text{where } r_t = \rho/(t+1)
			\end{align}
			\item If (QSA0) and (QSA5) hold, then $\barUpupsilon(\theta) = 0$ for each $\theta \in \Re^d$.
			
			\item If (QSA2) and (QSA3) hold, then $\{\ODEstate_t\}$ is \textit{ultimately bounded}: there exists $b<\infty$ such that for any $\ODEstate_0$, $\limsup_{t \to \infty} \|\ODEstate_t \| \leq b$.  
		\end{romannum}  
	\end{theorem}  
	
	\label{e:BigGlobalODEvanishing_thm}
	
\end{subequations}

\textit{Proof:}
The proof of part (i) is given at the end of Section B.3, while part (iii)  follows from \cite[Prop. 4.33 and 4.34]{CSRL}.

\begin{wrapfigure}[9]{r}{0.375\textwidth}
	\vspace{-.5em}
	\includegraphics[width= \hsize]{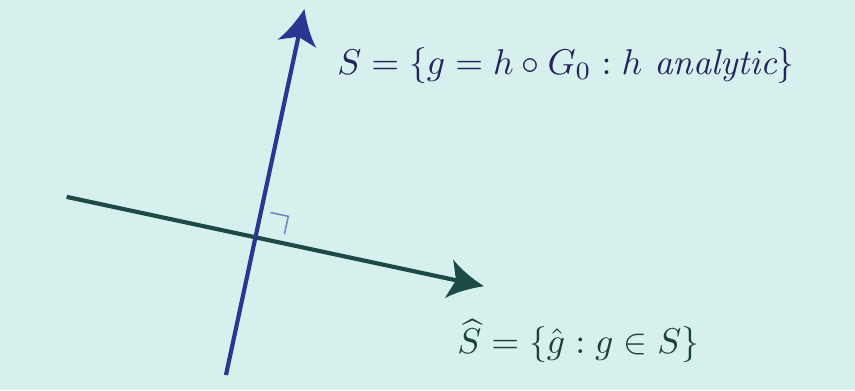}
	\caption{\textit{Hidden geometry}: orthogonality of the function classes $S$ and $\haS$. }
	\label{fig:orth}
\end{wrapfigure}

Part (ii) is based on the geometry illustrated in 	\Cref{fig:orth}:   $S$ denotes the set of functions $g\colon\Re^m\to\Re$ that are
analytic functions of the probing signal,   and $\haS$ the set of functions $h\colon\prstate\to\Re$ that have zero mean and solve Poisson's equation for some $g\in S$.
\Cref{t:hah-g-orth} establishes orthogonality of these two function classes   in $L_2$:    
$\langle g,h\rangle =0$  for each $g\in S$ and $h\in\haS$ (in the notation \eqref{e:spath-notation}),  which  indicates
\[
\lim_{T\to\infty} \int_0^T g(\qsaprobe_t) h(\Phi_t) \, dt = 0 \,, \ \  \text{for each $\Phi_0$}
\]
The definition \eqref{e:Upupsilon} gives 
\[
\Upupsilon_i(\theta,\Phi) =  -\sum_{j=1}^d \haA_{i,j}( \theta, \Phi)  f_j( \theta, \qsaprobe)\,,
\quad \text{  for each $\qsaprobe\in\Re^m$ and each $\Phi \in \prstate$.}
\]
Since $\haA_{i,j}(\theta,\varble)  \in \haS$ and $f_{j}(\theta,\varble)  \in S$ for each $i,j$ and $\theta$, it follows that $\barUpupsilon(\theta) = 0$. 
\qed

\subsection{Acceleration}
\label{s:fix_PR+FB}
Two approaches to obtain the convergence rate bounds $\| \ODEstate_T - \theta^\ocp\| = O(a_T^{2})$ are described here.

\wham{Polyak-Ruppert averaging}

Fast convergence can be obtained using standard averaging, but only under restrictive conditions.  In \Cref{t:Couple_main} we adopt the notation,
\begin{equation}
	\barY^* \eqdef  [A^\ocp]^{-1} \barUpupsilon^*,
	\text{ where } 
	\barUpupsilon^* =\barUpupsilon(\theta^*) \,,\qquad \haf^*_t \eqdef \haf(\theta^*,\Phi_t)\,,\ \ t\ge 0
	\label{e:barY}
\end{equation}
The proof of \Cref{t:Couple_main} is postponed to the supplementary material. 

\begin{theorem}
	\label[theorem]{t:Couple_main}
	Suppose that (QSA1)--(QSA5) hold, and the following additional assumptions are imposed:   $\{ \ODEstate_t\}$ is obtained using $a_t = (1+t)^{-\rho}$ with $\rho\in(1/2,1)$,  and estimates using PR averaging \eqref{e:PR_def} are obtained using  $T_0 =(1-1/{\kappa})T$ with $\kappa > 1$  fixed.   Then,
	\begin{subequations}
		\begin{align}
			\ODEstate_t &=  
			\theta^\ocp  +   a_t [\barY^*  -  \haf^*_t ] + o(a_t)
			\label{e:Couple_ODEstate} 
			\\ 
			\ODEstatePR_T &=  
			\theta^* + a_T [c(\kappa,\rho) + o(1)] \barY^*  + O(T^{-2\rho})
			\label{e:PRtheta}
			\\
			& \text{with } c(\kappa,\rho) = \kappa[1-(1-1/\kappa)^{1-\rho}]/(1-\rho)
			\nonumber
		\end{align}
		\label{e:newPR_rates}
	\end{subequations} 
\end{theorem}

\Cref{t:P-meanflow} and the definition \eqref{e:barY}  imply that $\barY^* = 0$ under (QSA0).     
Eq.~\eqref{e:PRtheta} of \Cref{t:Couple_main} then implies that PR averaging leads to a convergence rate bounded by $O(T^{-2\rho})$, subject to (QSA0)--(QSA5).

It was  discovered recently that the vector $\barUpupsilon^*$ also appears in representations of the SA recursion 		\eqref{e:SA_recur} when $\{\zeta_n\}$ is a smooth function of a Markov chain.  This introduces bias and  potentially large variance  
for fixed step-size algorithms  ($\alpha_{n+1} \equiv \alpha>0$), even when using averaging \cite[\S 2.5.3]{laumey22b},\cite{laumey22c}.

\wham{Forward-backward filtering}

We present a technique to achieve near quartic MSE convergence rate without imposing additional structure on the probing signal.   Motivation is two-fold:   first, it is hoped that this technique can be extended to SA for which there is no known analog to (QSA0);   second, in some applications of QSA it may not be possible to ensure that (QSA0) is satisfied.

Forward-backward (FB) filtering is defined by a pair of QSA ODEs:
\begin{align}
	\ddt \ODEstate^-_t = a_t f(\ODEstate^-_t,\qsaprobe^-_t) 
	\qquad
	\ODEstateFB_T = \half  [\ODEstatePR_T+\ODEstatePRm_T] 	
	\label{e:Backwards_QSA}
\end{align}
where $\qsaprobe^-_t \eqdef \qsaprobe_{-t}$,     $\ODEstatePR_T$ is defined in \eqref{e:PR_def},   and $\ODEstatePRm_T$ is defined analogously using $\{ \ODEstate^-_t\}$.


\begin{theorem}
	\label[theorem]{t:FB_ROC}
	If the assumptions of \Cref{t:Couple_main} hold and the frequencies $\{\omega_i\}$ in \eqref{e:qSGD_probe0} are distinct, then, $\ODEstateFB_T = \theta^{\ocp} + O(T^{-2\rho})$.  
\end{theorem}
\textit{Overview of proof:} The analogous vector   $\barY^*$ for the QSA ODE   \eqref{e:Backwards_QSA} is denoted $\barY^{*-}$.
The identity  $\barY^{*-} = -\barY^*$ is established through consideration of \eqref{e:barY}, and  the desired bound is then obtained by combining   \eqref{e:PRtheta} with the definition \eqref{e:Backwards_QSA}.  \qed

Assumption (QSA0) implies that the frequencies $\{\omega_i\}$ are irrationally related.   
This might appear far stronger than needed:  for one,  the full rank condition for $\Sigmaqsa$ can be achieved under far weaker assumptions.   In two dimensions,   $\Sigmaqsa =\half I$ for each of the following special cases:  $\qsaprobe _t = [\cos (   t) , \cos (  2 t) ]^\transpose $   and $ 
\qsaprobe _t = [\cos (  t) , \sin (  t) ]^\transpose$.     
However, in each case we can construct a second order polynomial $h\colon\Re^2\to\Re$ such that $h( \qsaprobe _t ) =0$ for all $t$,  so that the \textit{excitation} implied by the full rank condition is lost through a simple nonlinearity.    This is only the first sign of trouble.

We  consider next a numerical example  designed with four positive frequencies that are linearly independent over the rationals, denoted $[  \omega_{1} \,, \,	\omega_{2} \,, \, \omega_{3} \,, \,\omega_{4} ] $. 
Consider the linear QSA ODE,
\begin{align*}
	&\ddt\ODEstate_t =a_tf(\ODEstate_t,\qsaprobe_t) =(1+t)^{-\rho}[(A^*+A^\circ_t)\ODEstate_t +  10 b^\circ_t]
	\\[.5em]
	&	\text{with}  \quad
	A^\circ_t  = 
	\begin{bmatrix}
		4\sin(\omega_{1} t)&\sin(\omega_{2}t)\\
		\sin(\omega_{3}t)&4\sin(\omega_{4}t)\\
	\end{bmatrix} \qquad
	b^\circ_t =
	\begin{bmatrix}
		2\cos(\omega_{1} t)\\
		\cos(\omega_{4}t)\\
	\end{bmatrix}  \qquad \text{and $A^* $ Hurwitz. }
\end{align*}
We have $\barf(\theta) = A^* \theta$ and hence $\theta^\ocp =0$.   
To fit the notation \eqref{e:qSGD_probe0}
we must take $\qsaprobe_t\in \Re^6$,  which can be chosen so that   $\Sigmaqsa =\half I $.  
The expression $\barY^* = -20[2/\omega_{1}, 1/\omega_{4}]^\transpose$  follows from \eqref{e:barY}.

The conclusions of \Cref{t:FB_ROC} are illustrated in \Cref{fig:comp_rho_ROC}.
The model was simulated with  $A^* = - 0.8 I $
and frequencies  $ [ 	\pi 	\,, \, \sqrt{3} \,, \, 4 	\,, \, \sqrt{5}  ]/5$.   
The QSA ODEs were constructed using  three values of the parameter in the gain process,  $\rho \in \{0.6,0.7,0.8\}$.

The plots shown on the left in \Cref{fig:comp_rho_ROC}  were obtained using PR averaging with  $T_0 =(1-1/{\kappa})T$ and $\kappa = 4$.  
The convergence rate is $O(T^{-\rho})$ because $\barY^* \neq \Zero$  (recall \eqref{e:PRtheta}).
FB filtering achieves the convergence rate of $O(T^{-2\rho})$, for each value of $\rho$ tested,  as predicted by \Cref{t:Couple_main}.

\subsection{Quasi-Monte Carlo}
\label{s:QMCnum}

Results from a  simple experiment are provided here to illustrate that we can design algorithms to achieve convergence rates far faster than $O(\log(n)^K/n)$ in applications to QMC.

Suppose that $\bfqsaprobe$ is $m$-dimensional, with components equal to triangle waves:   $\qsaprobe_t^i  =  \bigtriangleup (\omega_i t + \phi_i)$ for each $i$ and $t$, with $\bigtriangleup$ the unit sawtooth wave with unit period and range $\pm 1$:
\[
\bigtriangleup (t) =  1-4 | \half - \text{frac}\{t+ \tfrac{1}{4} \} | 
\]

We present results from a simple experiment as illustration:
Our goal is to estimate the mean of $g(Y)$ with $g\colon\Re^2\to \Re$ and $Y$ uniformly distributed on the rectangle $[-1,1]^2$.   
The probing signal is chosen to be two-dimensional
, so that a QSA algorithm to estimate $\barh$ is obtained with $f(\theta,\qsaprobe) =-\theta +h(\qsaprobe)$  for which $\barf(\theta) = -\theta + \barh$:
\begin{equation}
	\ddt\ODEstate_t = a_t [  -\ODEstate_t  +   h(\qsaprobe_t) ]
	\label{e:QSAQMC}
\end{equation} 
Based on the formula \eqref{e:barY} we have $\barY^*=0$.

Consider the function $h(x_1,x_2) = \exp( \gamma x_1 )\sin(2\pi (x_2-x_1) ) $, for which $\barh=0$ for any value $\gamma$.  The frequencies $\{\omega_1,\omega_2\} = \{\log(6),\log(2)\} \approx \{     1.8    , 0.69 \}$ were chosen, along with several values of $\rho$, 
and $\kappa=5$ in application of PR averaging.   The ODEs were approximated using an Euler approximation with sampling time $T_s=0.1$ sec., which is roughly 1/5 of the shortest period $1/\omega_1$.

The data displayed in \Cref{f:Hist_IID+QSA_p7} is based on four experiments,  differentiated by two values   $\gamma=1,4$,  and choice of probing signal.    The first column uses the triangle wave described above, while the second is based on  $\{\qsaprobe_{t_n} : n\ge 0\}$ i.i.d.\ and uniform on $[-1,1]^2$,  where $\{t_n\}$ are the sampling times used in the Euler approximation for QSA.    It is known that  Polyak-Ruppert averaging shares the same CLT (asymptotic) variance as the usual sample path average when the probing signal is i.i.d.\  \cite{pol90,poljud92}.    The empirical variance using PR averaging was found to be similar to what was obtained using standard Monte Carlo.

Each histogram was created based on a runlength of $T=10^4$  (corresponding to $10^5$ samples, given $T_s=0.1$),   and 500 independent runs in which the phases were sampled uniformly and independently on $[0,1]$,   and the initial condition $\theta_0$ was   
sampled uniformly and independently on $[-25,25]$.  The value $\rho=0.7$ was chosen, resulting in $T^{2\rho} \approx 4\times 10^5$.  
The  MSE is roughly \textit{four orders of magnitude larger} for the stochastic algorithm as compared to QSA.

The remaining experiments surveyed here are based on $\gamma=1$.

\Cref{fig:PRQSA} shows histograms of the scaled estimation error for quasi-Monte Carlo with averaging for two cases: $\rho=0.7\,(\beta=1.4)$ and $\rho=0.8\,(\beta=1.6)$.   The scaled error exhibits more variability as $\rho$ is increased.  In this experiment,  the histogram obtained with  $\rho=0.8$ is roughly two times wider than for $\rho=0.7$.    In other words,  the rate of convergence is bounded by a constant $B_\rho$ times $T^{-2\rho}$,  and in these experiments we observe that the best constant $B_\rho$ is an increasing function of $\rho$.

\Cref{f:VerySorry} shows sample paths of estimates for   standard Monte Carlo and the stochastic PR algorithm,  with     $\{\qsaprobe_{t_n} : n\ge 0\}$ i.i.d.\ and uniform.

\begin{figure}
	\centering
	\includegraphics[width= 0.75\hsize]{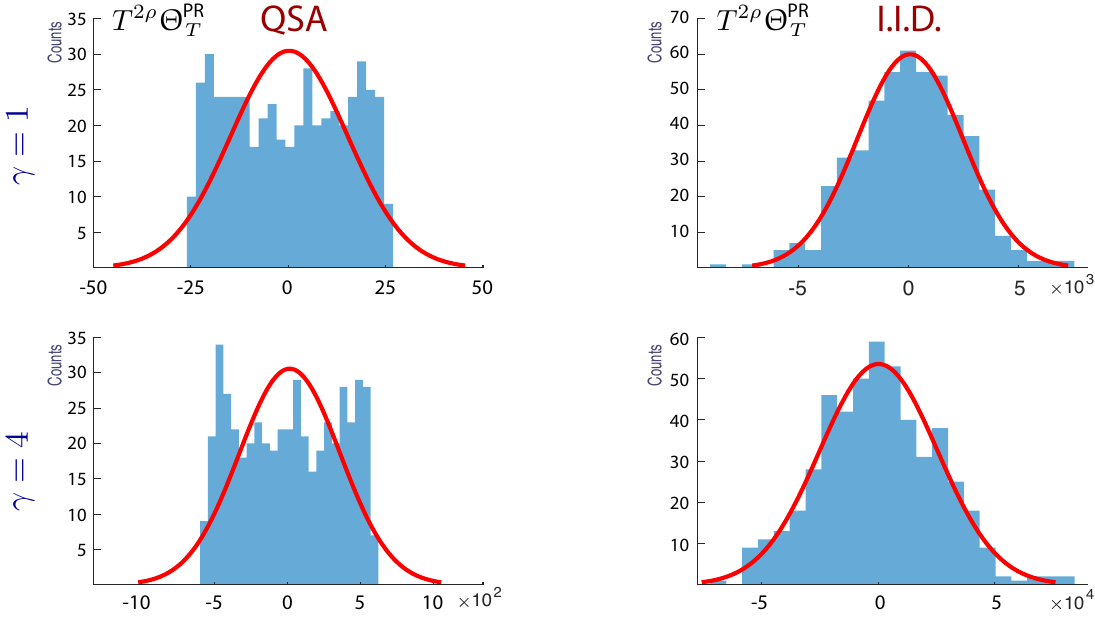}
	\caption{Comparison of Monte Carlo and quasi-Monte Carlo with Polyak-Ruppert averaging.}
	\label[figure]{f:Hist_IID+QSA_p7}
\end{figure}

\begin{figure}[htbp]
	\centering
	\includegraphics[width= 0.8\hsize]{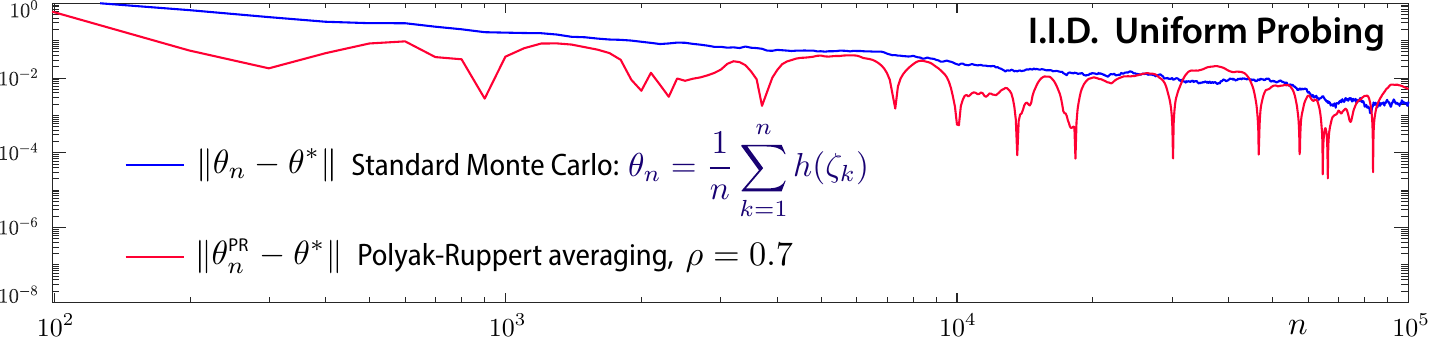}
	\caption{Slow convergence of Monte Carlo and  Monte Carlo with Polyak-Ruppert qveraging.}
	\label[figure]{f:VerySorry}
\end{figure}

The next  set of experiments illustrate that the rate of convergence observed in \Cref{f:VerySorry} is \textit{far slower} than observed in any of the deterministic algorithms.

Each row of \Cref{f:notSorry} contains four plots:  $| \ODEstate_T  |$  obtained with \eqref{e:QSAQMC} using the indicated value of $\rho$,     $| \ODEstatePR_T  |$ obtained using PR averaging,   and $T^{-\rho}$, $T^{-2\rho}$ as these are the convergence rate bounds for each case.     The plots illustrate that the $O(T^{-2\rho})$ bound on the convergence rate using PR averaging is achieved for smaller values of $\rho>1/2$.   With $\rho=0.9$ the results are not so clearly compatible with theory;
recall that the theory requires $\rho<1$, so we expect numerical challenges when $\rho$ is close to~$1$.

\begin{figure}[htbp]
	\centering
	\includegraphics[width= 0.85\hsize]{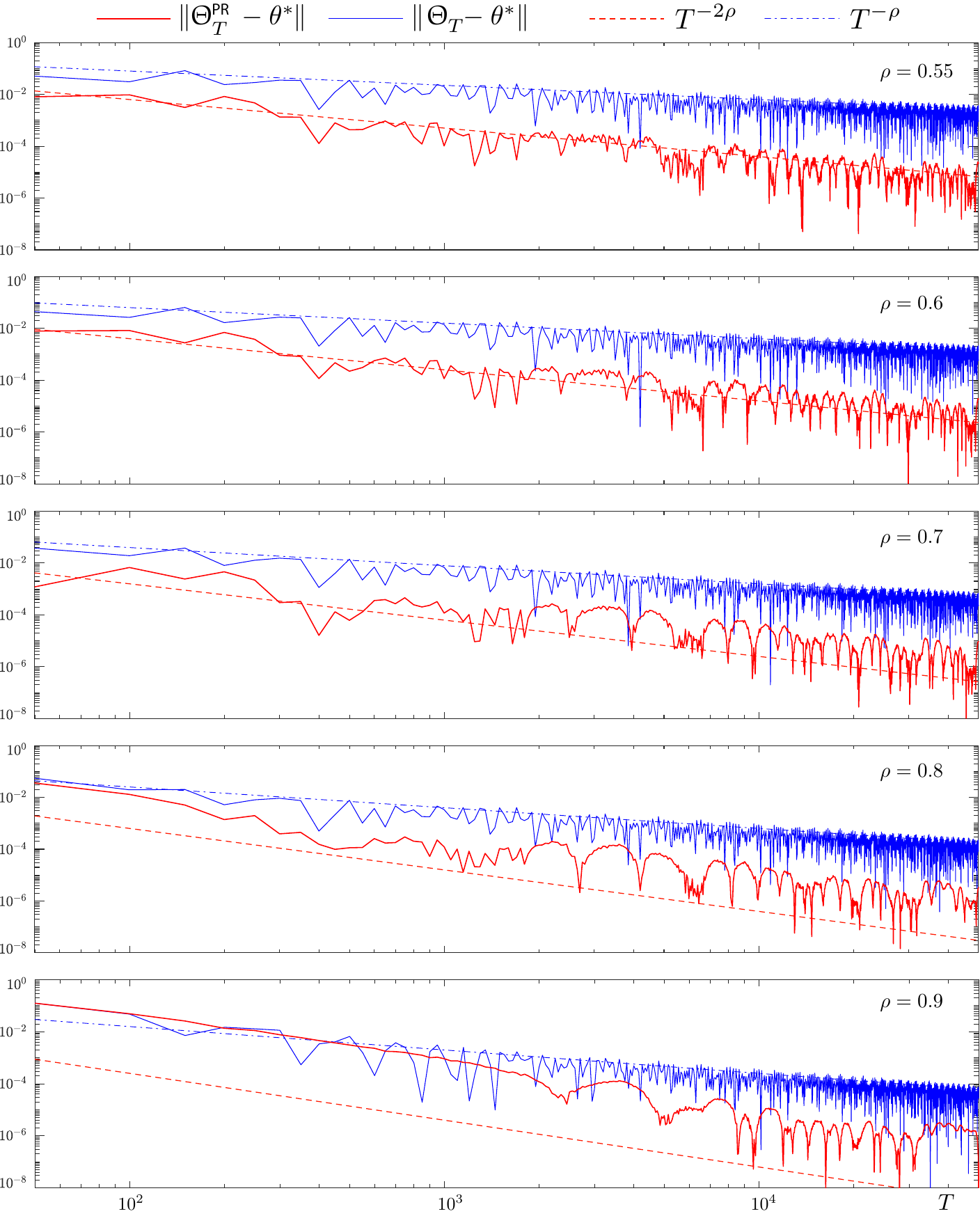}
	
	\caption{Quasi-Monte Carlo with Polyak-Ruppert averaging.}
	\label[figure]{f:notSorry}
\end{figure}

\vfill

\clearpage

\section{Gradient-Free Optimization} 
\label{s:GFO}

We now turn to   applications to gradient-free optimization (GFO).    
It is assumed that the objective   $\Obj \colon\Re^d\to\Re$ is $C^2$ and that it has a unique minimizer denoted $\thetaopt$,
and that  $\Sigmaqsa \eqdef \langle\qsaprobe\qsaprobe^\transpose \rangle $ is full rank.

\subsection{Gradient-Free Optimization and QSA}
\label{s:GFO_intro}
 
Two approaches are considered in the following:  For $\epsy>0$ and each $\ODEstate_0 \in \Re^d$,  
\begin{align}
	\textbf{$\mathbf{1}$qSGD:} \qquad	&	\fq(\ODEstate_t,\qsaprobe_t) = -\frac{1}{\epsy}\qsaprobe_t \Obj(\ODEstate_t+\epsy\qsaprobe_t)
	\label{e:qSGD1}
	\\
	\textbf{$\mathbf{2}$qSGD:} \qquad	&\fqq(\ODEstate_t,\qsaprobe_t) = -\frac{1}{2\epsy}\qsaprobe_t [\Obj(\ODEstate_t+\epsy\qsaprobe_t)-\Obj(\ODEstate_t-\epsy\qsaprobe_t)]
	\label{e:qSGD3}
\end{align} 

The mean vector fields $\barf^{\,\text{1Q}},\barf^{\,\text{2Q}}$ are identical, provided the probing signal   is symmetric:
\begin{proposition}
	\label[proposition]{t:samefbar}
	Suppose (QSA1)-(QSA5) hold, that $\Obj$ is $C^2$ with unique minimizer $\thetaopt$. Assume moreover that the   probing signal is  of the form \eqref{e:qSGD_probe} with $m =d$ and distinct frequencies $\{\omega_i\}$. Then, the average vector fields for $1$qSGD and $2$qSGD are equal:
	\begin{equation}
		\barf   (\theta) =  
		\langle\fq(\theta,\qsaprobe)\rangle  = \langle\fqq(\theta,\qsaprobe)\rangle = 
		- \Sigmaqsa \nabla \Obj(\theta) + O(\epsy^2)
		\label{e:barf_qsgd}
	\end{equation}
\end{proposition}

For the special case of a strongly convex  objective $\Obj$,  bounds on bias of $2$qSGD are well known \cite{spa03}.    An application of \Cref{t:samefbar} implies that the bias for $1$qSGD is identical:   
\label{s:GFO_bias_var}
\begin{corollary}
	\label[corollary]{t:bias_qSGD1_strconvx}
	Suppose the assumptions of \Cref{t:samefbar} hold, and $\Obj$ is strongly convex.
	\\
	Then, $\| \theta^{\ocp}  -\thetaopt \|  \leq O(\epsy^2)   
	$   for either $1$qSGD or $2$qSGD.
\end{corollary}

		
\subsection{Numerical Examples}
\label{s:GFO_num_ex}

We survey results of numerical experiments, whose objective functions were selected from \cite{simulationlib}.\footnote{Publicly available code obtained under GNU General Public License v2.0.} 
For a rectangular region $B_0 \subseteq \Re^d$, we will frequently use the expression ``projection of sample-paths onto $B_0$''. This means we project the trajectories $\{\ODEstate_t\}$ to $B_0$ component-wise in any approximation of the QSA ODE.   For the one-dimensional setting with $B_0 = [-1,1]$,  the projection is defined by 
$ \max\{-1,\min\{1,\theta_{n}\}\}$ where $\theta_n$ is an approximation of $\ODEstate_{t_n}$ at sampling time $t_n$.

Projection is often necessary because Lipschitz continuity of $f$ is assumed in (QSA2). This requires $\Obj$ to be Lipschitz continuous when employing $1$qSGD \eqref{e:qSGD1} and $\nabla \Obj$ when using $2$qSGD \eqref{e:qSGD3} \cite{chedevbermey21,CSRL}. Restricting $\{\ODEstate_t\}$ to a closed and bounded set is a way to relax these requirements.

Each experiment contained the following common features:
\begin{romannum}
	\item Only $1$qSGD was considered.
	
	\item The gain process was $a_t = \min\{a_0,(t+1)^{-\rho}\}$ with $1/2 \leq \rho \leq 1$, and $a_0>0$.
	
	\item The QSA parameters $a_0$ and $\epsy$ were problem specific and chosen by trial and error.

	\item The ODE \eqref{e:qSGD1} was approximated by an Euler scheme with sampling time equal to $1$~sec. This crude approximation is justified via $(\text{ii})$ by choosing $a_0 > 0$ sufficiently small.
	
	\item Averaging was performed with $\kappa=5$ (final 20\%\ of samples).

	\item A rectangular constraint region $B_0 \subseteq \Re^d$ was fixed. The selection of $B_0$ was based on conventions of the particular objective given in \cite{simulationlib}.
	
	\item For each objective and algorithm, the frequencies were held fixed in $M$ independent experiments:  
	$\{\omega_i : 1 \leq i \leq d\}$ were  uniformly sampled from $[0.05,0.5]^d$.   The initial conditions and phases were not held constant:
	For $\{1 \leq m \leq M\}$,
	\begin{alphanum}
		\item The phases $\{\phi_i^m : 1\le i \le d \}$ were sampled uniformly on $[-\pi/2,\pi/2]^d$.  
		That is, the probing signal respected $\eqref{e:qSGD_probe}$, and in the $m$th experiment the probing signal took the form 
		\[
		\qsaprobe^{m}_t = 2[\sin(\omega_1t+\phi_1^m) \,,\cdots, \ \sin(\omega_dt+\phi_d^m)]^\transpose
		\] 
		giving $\Sigmaqsa =2I$. 
		\item The initial condition $\ODEstate^m_0$ were selected uniformly at random from $B_0$.
	\end{alphanum}

	\item The outputs were the sample paths $\{\ODEstate^m_t: t \geq 0\}$, $\{\Obj(\ODEstate^m_t): t \geq 0\}$ and/or the sample covariance:
	\begin{equation}
		\barSigma_T = \frac{ 1}{M} \sum_{i=1}^{M} \ODEstate_T^i  {\ODEstate_T^i }^\transpose   -  \barODEstate_T \barODEstate_T^\transpose\,,\qquad \barODEstate_T= \frac{ 1}{M} \sum_{i=1}^{M} \ODEstate_T^i
		\label{e:empCov}
	\end{equation}
	
\end{romannum}

Theory predicts that root mean square error $T^{2\rho} \sqrt{\text{tr}(\barSigma_T)}$ will be bounded in $T$ subject to conditions on the algorithm and probing signal.
It is found that the results remain consistent with theory for moderate dimension in each example tested.

		\begin{figure}[h]
			\includegraphics[width=1\hsize]{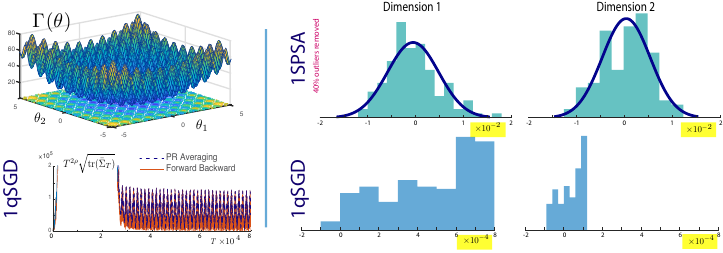}
			\caption{Rastrigin objective (top left), scaled trace of empirical covariance (bottom left), histograms of estimation error for $1$SPSA with PR averaging (top middle and top right), histograms of estimation error for $1$qSGD with PR averaging (bottom middle and bottom right).}\label{fig:Rastrigin}
		\end{figure}

\wham{Rastrigin:}    The experiments illustrated in \Cref{fig:Rastrigin,fig:contour_rasta} are based on the Rastrigin objective with $d=2$ and projection region $B_0 = [-5.12,5.12]^d$, following \cite{simulationlib}.
The experiments used $T= 8 \times 10^4$, $M=200$ and $a_0 =0.5$. The 1qSGD and 2qSGD algorithms were compared in this standard benchmark \cite{simulationlib},  for which a plot of the objective can be found on the upper left in \Cref{fig:Rastrigin}.    In each case the values $\epsy = 0.25$, $\kappa = 5 $  and $\rho=0.85$ were used.  
The stochastic counterparts were included in these experiments in which the exploration sequence $\{\zeta_n\}$ was chosen i.i.d.\ and zero-mean. It was a scaled and shifted Bernoulli with $p=1/2$ and support chosen so that its covariance matrix $\Sigma$ satisfied $\Sigma=\Sigmaqsa$.

Histograms are shown only for $1$qSGD and $1$SPSA using PR-averaging in \Cref{fig:Rastrigin}.   The results obtained for $1$qSGD with FB filtering were similar.    Results obtained using $2$qSGD and $2$SPSA were similar qualitatively.  Outliers were identified with Matlab's \textit{isoutlier} function and  excluded in these plots.   Outliers were found in 20\%\ of the independent runs for $1$SPSA,   and none for $1$qSGD.   Outliers for deterministic algorithms were observed in other experiments, but fewer than in their stochastic counterparts. 

		\begin{wrapfigure}{r}{0.35\textwidth}
	\includegraphics[width= \hsize]{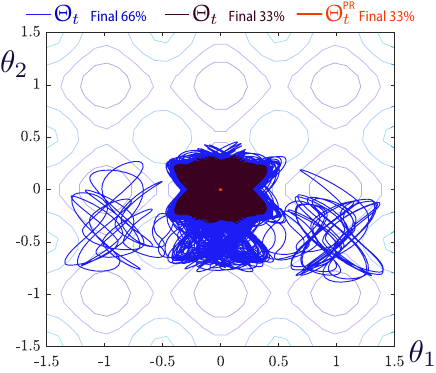}
	\caption{1qSGD    for the Rastrigin objective:
		$\{\ODEstate_T \}$ for the final $66\%$ of the run, and $\{\ODEstatePR_T\}$ for the final $33\%$. }
	\label{fig:contour_rasta}
\end{wrapfigure}

As predicted by theory,  the   root-MSE  $  T^{2\rho} \tr(\widebar{\Sigma}_{T})^\half$ is bounded in $T$ when using FB filtering or PR averaging:  see the plot on the lower left in  \Cref{fig:Rastrigin}. Apart from achieving near quartic convergence rates, the estimates of $\thetaopt$ resulting from the deterministic algorithm had less variability than their stochastic counterpart. This reduction of variability was roughly two orders of magnitude. In both cases, the histograms show that the algorithms consistently overestimate $\thetaopt$. This is most likely due to the bias that is inherent in these algorithms (recall \Cref{s:GFO_intro}).

\Cref{fig:contour_rasta} shows part of the trajectory of $\{\ODEstate_T\}$ for a short run for better vizualization. We see that $\{\ODEstate_T\}$ oscillates between saddle points and local extrema before settling around the minimizer $\thetaopt =0$ near the end of the run.  The corresponding PR-estimates   very closely approximate $\thetaopt$ for the final 33\%\ of the run.

  \begin{figure}[h]
	\centering
	\includegraphics[width= 1\hsize]{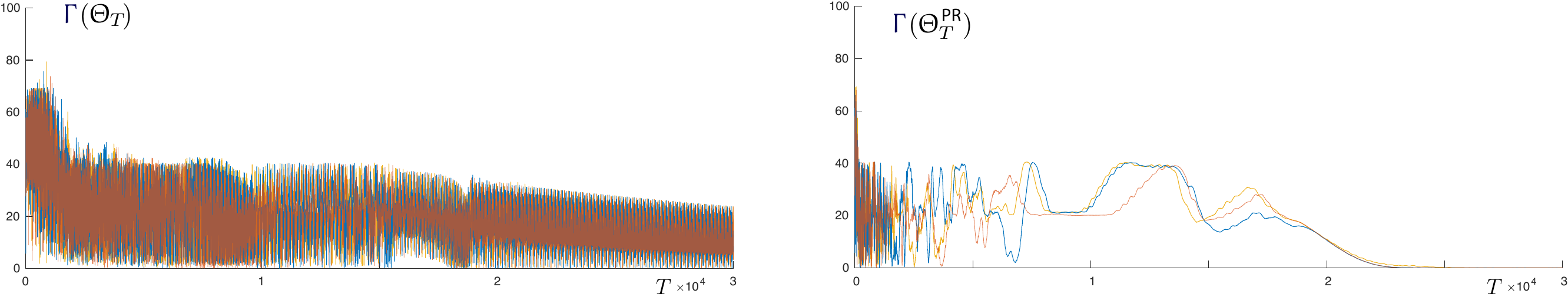}
	\caption{Evolution of $\Obj(\ODEstate_T)$ (left) and $\Obj(\ODEstatePR_T)$ (right) as functions of $T$ for the Rastrigin objective.}
	\label[figure]{fig:Rasta_Loss}
\end{figure}

  The plots in \Cref{fig:Rasta_Loss} show a comparison of both $\Obj(\ODEstate_T)$ and $\Obj(\ODEstatePR_T)$ for three distinct initial conditions $\ODEstate_0$. This plot illustrates the benefits of PR averaging as $\Obj(\ODEstatePR_T)$ approaches $\Obj(\theta^*)=0$ much quicker than $\Obj(\ODEstate_T)$ for each initial condition.


\begin{figure}[h]
	\centering
	\includegraphics[width= 1\hsize]{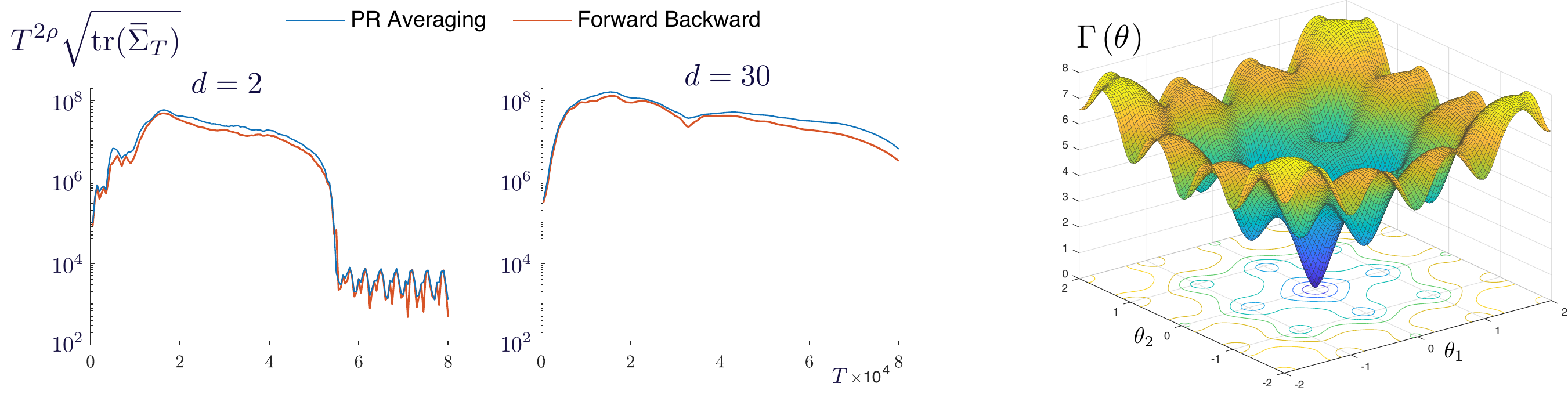}
	
	\caption{Scaled root-MSE for $d=2$ (left) and $d=30$ (middle). At right is the Ackley objective for $d=2$.}
	\label[figure]{f:Ackley}
\end{figure}
 
\wham{Ackley:}  This is another standard benchmark \cite{simulationlib},  whose objective function  is shown on the right in  \Cref{f:Ackley} in the special case $d=2$. We performed $M=50$ independent experiments of run length $T= 8 \times 10^4$ to minimize the Ackley objective with projection region equal to $B_0 = [-32.768,32.768]^d$ \cite{simulationlib}. Experiments only employed $1$qSGD and were performed for two sets of parameters:
\begin{romannum}
\item $\rho=0.85$, $a_0= 0.02$ and $\epsy = 0.01$.
\item $\rho=0.7$, $a_0= 0.07$ and $\epsy = 0.1$.
\end{romannum}

For each set of parameters, the root mean square error $T^{2\rho} \sqrt{\text{tr}(\barSigma_T)}$ was obtained for $d=2$ and $d=30$. Results for case (i) are pictured in \Cref{f:Ackley} and repeated in \Cref{f:Ackley_08507_cov} along with results for case (ii). Both PR averaging and FB filtering are successful in achieving  $O(T^{-2 \rho})$ convergence rates. 

Simulations indicate that the variance is bounded by $O(a_T^4)$ in each case, but it is possible that their value grows with dimension: we observe much larger values of   $\tr(\widebar{\Sigma}_{T})^\half$ at the end of the run for  $d=30$.   It is possible that better results will be improved by with a different choice  of probing signal $\bfqsaprobe$.

  \begin{figure}[h]
	\centering
	\includegraphics[width= 1\hsize]{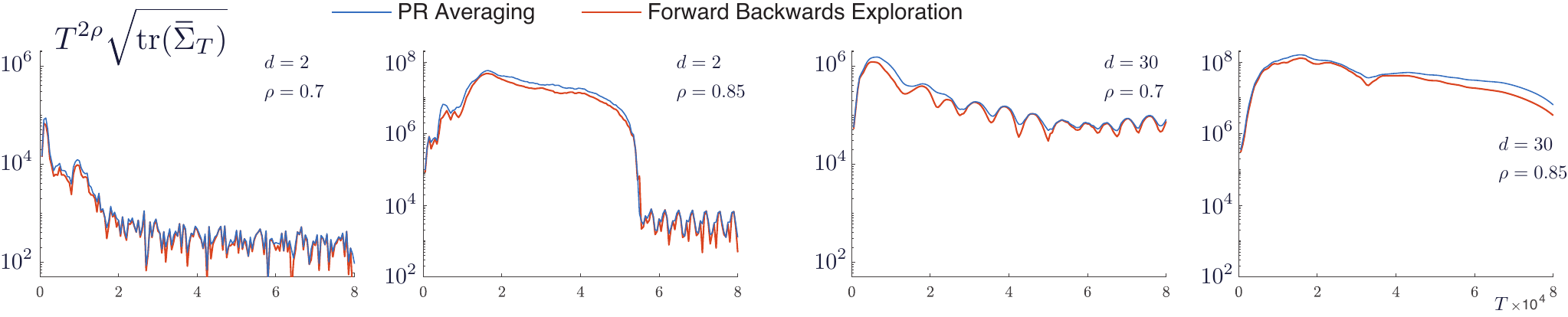}
	\caption{Comparison of scaled empirical covariances for Ackley objective.}
	\label[figure]{f:Ackley_08507_cov}
\end{figure}

\begin{figure}[h]
	\centering
	\includegraphics[width= 1\hsize]{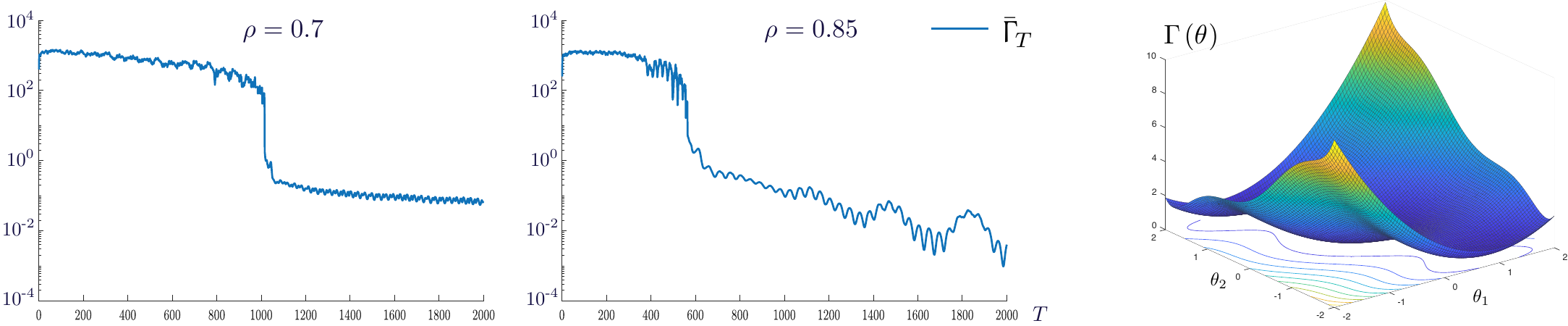}
	
	\caption{Average of $\Obj(\ODEstate^i_T)$ over $1\le i\le M$ independent runs for the three-hump camel objective. 
	  } 
	\label[figure]{f:3humpObj}
\end{figure}

\wham{Three-Hump Camel:}  This standard $d=2$ dimensional benchmark, plotted on the right in \Cref{f:3humpObj},  uses $B_0 = [-5,5]^d$ \cite{simulationlib}.

Results from the $1$qSGD algorithm are surveyed here, for the values  $\rho=0.85$ and $\rho=0.7$;  in each of these two cases the values $a_0= 0.01$ and $\epsy = 0.5$ were successful.  

The average loss across initial conditions was obtained via
\[
\bar{\Obj}_T = \frac{1}{M} \sum^M_{m=1} \Obj(\ODEstate^m_T)
\]
with estimates $\{\ODEstate^m_t : t\ge 0\}$ obtained using PR-averaging.
Its evolution in $T$ is plotted on the left of \Cref{f:3humpObj} where it is possible to observe that the value $\Obj(\ODEstatePR_t)$ converges rapidly towards its optimal value of zero. 

Covariance estimates were obtained based on $M=50$ independent experiments with run length $T= 8 \times 10^4$ and the root mean square error $T^{2\rho} \sqrt{\tr(\barSigma_T)}$ is shown by \Cref{f:3Hump_2dim08507cov}. The results for this experiment agree with the previous observations: near quartic rates are achieved. The variance appears to grow with $\rho$.

  \begin{figure}[htbp]
\centering
	\includegraphics[width= 1\hsize]{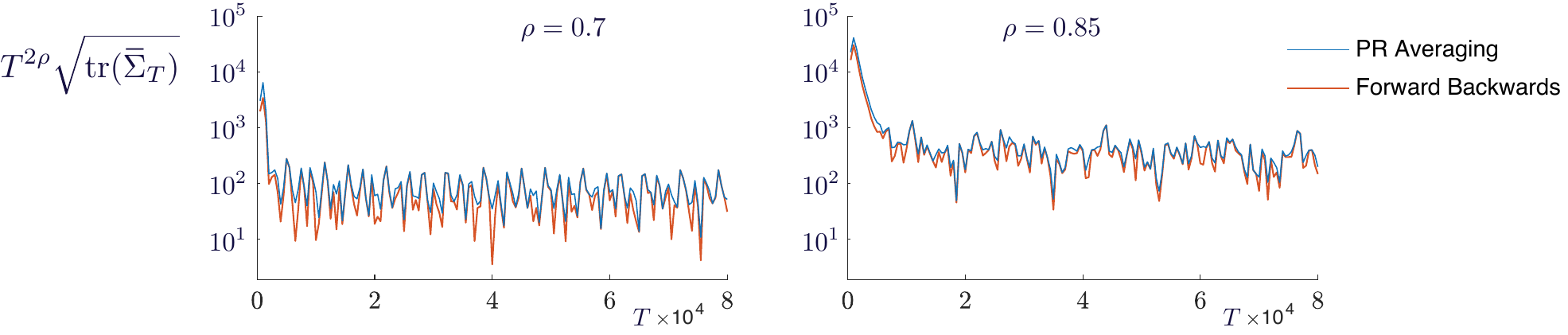}
\caption{Scaled empirical covariances for $\rho=0.7$ (left) and $\rho=0.85$ (right) for three-hump camel objective.}
\label[figure]{f:3Hump_2dim08507cov}
\end{figure}


\section{Conclusions} 
		\label{s:conc}

While it is convenient to design exploration around i.i.d.\ signals,  and this approach opens the doors to many powerful tools from probability theory,  we have shown that   deterministic ``noise'' has significant benefits.   Convergence rates can be accelerated dramatically provided the algorithm and deterministic probing signals are chosen with care.

There are many avenues for future research:

\witem
The impact of dimension on convergence rate appears to be understood for  QMC \cite{asmgly07}.   We currently do not know how to extend this theory to QSA because our analysis is rooted in properties of the $K$-dimensional clock process $\bfPhi$, which is far removed from the QMC setting.  

\witem
The optimal convergence rate for QSA is unknown for the approaches described in this paper, and we currently lack universal bounds that are not restricted to a particular algorithm.
\begin{romannum}
	\item   Can we improve the upper bounds for PR averaging and the FB algorithm?  It may be possible to achieve the bound $O(T^{-1-\rho})$,  rather than $O(T^{-2\rho})$ in current theory. 
	
	\item  Can we find fundamental bounds for any algorithm?    \textit{We haven't ruled out the creation of an algorithm with rate of convergence of order $O(T^{-10})$.   }
	This question is motivated by a long history of success in the stochastic approximation literature \cite{chu54o,rup88,pol90,poljud92,poltsy90}.

\end{romannum}

\witem 
It is not clear that the constraint $\rho>\half$ is required in this deterministic setting, and may be removed if we can improve the bounds in our analysis.
We have found in some experiments that the use of PR averaging results in very fast convergence even when this constraint is violated.

\witem Under what conditions is $  \sqrt{\text{tr}(\barSigma_T)}$ strictly smaller using the FB algorithm as compared to PR averaging?
We find in experiments that the FB algorithm usually outperforms averaging,  but
we have yet to find tools to obtain bounds that are rich enough to compare the two approaches.

\clearpage
	
	\def\cprime{$'$}\def\cprime{$'$}


	


	\appendix

	\addcontentsline{toc}{section}{Appendices}
\clearpage
	
\section{Assumptions for QSA Theory}
\label{s:assump}
The assumptions imposed in \cite[\S 4.5, 4.9]{CSRL} are listed below.   The setting there is far more general than here, since the entries of the probing signal are not restricted to functions of sinusoids.   In this prior work it is assumed that the probing signal   is a function of a deterministic signal $\bfPhi$ of the form $\qsaprobe_t = G(\Phi_t)$,  where $\bfPhi$ is   the state process for a dynamical system,
\begin{equation}
	\ddt\Phi_t = \qsaDyn(\Phi_t) 
	\label{e:qsaDynamics}
\end{equation} 
with $\qsaDyn\colon \prstate \to \prstate$ and $\prstate$ is a compact subset of the Euclidean space. It is assumed that it  has a unique invariant measure $\uppi$  on $\prstate$.

For the special case treated here, with $\Phi_t^i =   \exp(2\pi j[\omega_i t + \phi_i])$ for each $i$ and $t$,
we have $\prstate = S^K\subset \Co^K$ with $S$ the unit circle in $\Co$.   The dynamics \eqref{e:qsaDynamics} and the function $G$
are: 
\begin{equation}
	\begin{aligned}
		\ddt \Phi_t &= W \Phi_t   \qquad  &&    W \eqdef  2\pi j\text{\rm diag} (\omega_i)   
		\\
		\qsaprobe_t & = G(\Phi_t ) \qquad &&  G(z) \eqdef G_0(\half (z+ z^{-1})) \,,\qquad z\in \Co^K
	\end{aligned}
	\label{e:qsaDynLinear}
\end{equation}
where $G_0$ is the function appearing in \eqref{e:qSGD_probe0}. The function $G$ is then analytic on $z \in \{\Co \setminus \{0\}\}^K$ if $G_0$ is analytic on $\Co^K$. 

The proof of \Cref{t:arcsinLaw} below begins with a proof that $\uppi$  exists, with density    
$  \upvarrho(z_1,\dots, z_K)  =\prod_{i=1}^K u(z_i)$, where $u$ denotes the uniform distribution on $S$.   In particular, since $\qsaprobe_t = G(\Phi_t)$ for some function $G$,  the function $\barf$ in \eqref{e:ergodicA1} can be expressed,
\begin{equation}
	\barf(\theta)  =   \int_\prstate f(\theta,G(z))   \,  \uppi(dz) 
	\label{e:barf2}
\end{equation}

For any function $h: \Re^d \times \prstate \to \Re$ that is $C^1$, we define
\begin{equation}
	\clD h\, (\theta,z)  \eqdef  a_t[\clD^f h](\theta,z)   
	+   \partial_z   h \, (\theta, z) \cdot W  z  \, ,  \quad (\theta, z) \in \Re^d \times \prstate  
	\label{e:diff_gen}
\end{equation}
where $\clD^f h$ is defined in \eqref{e:direc_derivat}. The function $\clD h$ is continuous and the functional $\clD$ is known as the \textit{differential generator} in the Markov literature. Upon denoting $g = \clD h $, the chain rule gives 
\begin{equation}
	g(\ODEstate_t,\Phi_t) = \ddt h(\ODEstate_t,\Phi_t)
	\label{e:Chainrule_diffgen}
\end{equation}

The remaining assumptions are listed below.  Lipschitz bounds on $\haf$
and those in (QSA5) are partially justified by \Cref{t:BakerCor} subject to smoothness assumptions on $f$.

The first assumption sets restrictions on frequencies.

\begin{subequations}
	
	\wham{(QSA0a)}     
	$\qsaprobe_t  = G_0(\qsaprobe _t^0  )$ for all $t$,   with $\qsaprobe _t^0$ defined in \eqref{e:qSGD_probe0}.  The function 
	$G_0\colon\Re^K \to\Re^m$ is assumed to be analytic, with the coefficients in the Taylor series expansion for $G_0(\qsaprobe _t^0  )$ absolutely summable.

	\wham{(QSA0b)}   
	The frequencies $\{\omega_1\,,\dots\,,  \omega_K\}$ are chosen of the form
	\begin{equation}
		\begin{aligned}
			&\omega_i   = \log(a_i/b_i) > 0\,, 
			\ \  
			1\le i\le K\,, 
			\\
			&  \text{$\{\omega_i  \}$ ,  linearly independent over the rationals,}
		\end{aligned} 
		\label{e:logFreq}
	\end{equation}
	and with   $\{a_i,b_i\}$   positive integers.

	\wham{(QSA1)} The process $\bfma$ is non-negative, monotonically decreasing, and  
	\begin{equation}
		\lim_{t\to\infty} a_t = 0,\qquad\int_0^\infty a_r\, dr = \infty.
		\label{e:QSA_A5}
	\end{equation} 
	
	\wham{(QSA2)}
	The functions $\barf$ and $f$ are Lipschitz continuous:  for   a constant $\Lip_f  <\infty$,
	\begin{align*} 
		\|\barf(\theta') - \barf(\theta)\| &\le \Lip_f \|\theta' - \theta\|, 
		\\
		\|f(\theta',\qsaprobe) - f(\theta,\qsaprobe)\| + \|f(\theta,\qsaprobe') - f(\theta,\qsaprobe)\| 
		&\le
		\Lip_f [ \|\theta' - \theta\| +   \| \qsaprobe'-\qsaprobe\|   ]   \,, 
		\quad \theta', \, \theta\in\Re^d\,,  \qsaprobe, \qsaprobe' \in\Re^m 
	\end{align*}

	\def\SetV{\text{\sf S}}
	
	\wham{(QSA3)}
	The ODE $\ddt \odestate_t  = \barf  ( \odestate_t  )$ is globally asymptotically stable with unique equilibrium $\theta^\ocp$. Moreover, one of the following conditions holds:
	\begin{alphanum}
		\item  There is a Lipschitz continuous Lyapunov function $V: \Re^d \to \Re_+$, a constant $\delta_0>0$ and a compact set $\SetV$ such that $\nabla V (\odestate_t) \cdot \barf(\odestate_t)  \le -\delta_0 \| \odestate_t\|$ whenever $\odestate_t \notin \SetV$.
		\item  The scaled vector field $\barf_\infty\colon\Re^d\to\Re^d$ defined by $\barf_\infty(\theta) \eqdef \lim_{c \to \infty}  \barf(c \theta)/c $,  $\theta \in \Re^d$,  exists as a continuous function.    Moreover, the ODE@$\infty$ defined by  
		\begin{equation}
			\ddt x_t = \barf_\infty(x_t)
			\label{e:ODEatinfty}
		\end{equation}
		is globally asymptotically stable  \cite[\S 4.8.4]{CSRL}.
	\end{alphanum}

	\wham{(QSA4)}
	The vector field $\barf$ is   differentiable,  with derivative denoted 
	\begin{equation}
		\barA(\theta)  = \partial_\theta \barf \, (\theta)
		\label{e:barfDer}
	\end{equation}
	That is,   $\barA(\theta)$ is a $d\times d$ matrix for each $\theta\in\Re^d$, with 
	$\displaystyle
	\barA_{i,j}(\theta)  = \frac{\partial}{\partial \theta_j}   \barf_i\, (\theta)
	$.
	
	Moreover, the derivative  $\barA$ is Lipschitz continuous, and      $\barA^\ocp = \barA (\theta^\ocp)$ is Hurwitz.

	\wham{(QSA5)}
	$\bfPhi$ satisfies the following ergodic theorems for the functions of interest, for each initial condition $\Phi_0\in\prstate$:
	\begin{romannum}
		\item   For each $\theta$ there exists a solution  $\haf (\theta,\varble)$ to \textit{Poisson's equation} with forcing function $f$.   That is,
		\begin{equation}
			\haf( \theta, \Phi_{t_0} ) = 
			\int_{t_0} ^{t_1} [ f( \theta ,  \qsaprobe_t) - \barf(\theta)]  \, dt    + \haf( \theta, \Phi_{t_1} )   \,,
			\quad 0\le t_0\le t_1
			\label{e:PoissonA1}
		\end{equation}
		with $\barf$ given in \eqref{e:barf2} and for each $\theta$, $   \int_\prstate \haf(\theta,z) \,  \,  \uppi(dz)    = \Zero $.
		Finally, $\haf$ is continuously differentiable ($C^1$) on $\Re^d \times \prstate$. Its Jacobian with respect to $\theta$ is denoted
		\begin{align}
			&\haA(\theta,z) \eqdef \partial_\theta\haf(\theta,z)  \label{e:haA_def} 
			\\ 
			\text{where } 
			\int_\prstate &\haA(\theta,z) \,  \,  \uppi(dz)   = \Zero\, 
			\quad 
			\text{for each } \theta \in \Re^d
			\label{e:Ameanzero}
		\end{align}
		
		\item  For each $\theta$, there are $C^1$ solutions to  Poisson's equation with forcing functions $ \haf$ and $\Upupsilon$. They are denoted $\hahaf$ and $\haUpupsilon$, respectively, satisfying
		\begin{align}
			\hahaf( \theta, \Phi_{t_0} ) &= 
			\int_{t_0} ^{t_1}  \haf( \theta ,  \qsaprobe_t)  \, dt    + \hahaf( \theta, \Phi_{t_1} ) 
			\label{e:NoiseIntInt} 
			\\
			\haUpupsilon( \theta, \Phi_{t_0} ) &= 
			\int_{t_0} ^{t_1}  [\Upupsilon( \theta ,  \Phi_t) - \barUpupsilon(\theta) ] \, dt    + \haUpupsilon( \theta, \Phi_{t_1} ) \,,
			\quad 0\le t_0\le t_1 
			\label{e:haUpupsilon} 
			\\
			\text{with} 
			\quad 
			\barUpupsilon(\theta) & = \langle\Upupsilon(\theta,\Phi) \rangle=  - \int_\prstate \haA(\theta,z) f(\theta,G(z))\,  \,  \uppi(dz) 
			\label{e:barUpupsilon}
		\end{align}
		Moreover, for each $\theta$, 
		\begin{equation}
			\int_\prstate \hahaf(\theta,z) \,  \,  \uppi(dz) = 
			\int_\prstate \haUpupsilon(\theta,z) \,  \,  \uppi(dz)    =
			\Zero
			\nonumber 
		\end{equation}
	\end{romannum}
	
\end{subequations}

\begin{proposition}
	\label[proposition]{t:OneFrequency}  
	For any integer $n\ge 2$ there is a polynomial $\varrho$ satisfying
	\[
	\varrho(  \cos(   r)  )    =   \cos(nr)  \,,\quad r \ge 0 
	\]

	Consequently, in the special case of (QSA0b) with $K=1$,   for any increasing sequence of positive integers $n_1=1<n_2<\cdots<n_{K_\bullet}$ there is a polynomial function $G_0\colon\Re^K\to\Re^m$ such that
	\[
	\qsaprobe _t  =  G_0 (\omega_1 t  + \phi_1) 
	=   \sum_{i=1}^{K_\bullet} v^i \cos(2 \pi [n_i\omega_1 t + n_i\phi_1])  
	\]
	with $\{v^i\}$ as in 	\eqref{e:qSGD_probe}.    
\end{proposition}

\Cref{t:OneFrequency}  tells us that if we choose a small value for $K$ in (QSA0) then we must impose stronger conditions on the values of phases for the mixture of sinusoids model 	\eqref{e:qSGD_probe}.

\section{Technical Proofs} 
\label{s:tech_proofs}

	\subsection{Tighter Bounds for Quasi-Monte Carlo}

\Cref{t:BakerCor} below justifies boundedness of the integral \eqref{e:Poissoneqdef} and much more.  
Through careful design of the frequencies $\{\omega_i\} $ appearing in \eqref{e:qSGD_probe0} we can apply refinements of Baker's Theorem, as surveyed in the monograph \cite{bug18}.

\begin{theorem}
	\label[theorem]{t:BakerCor}   
	Suppose that the function 
	$h\colon\Re^m\to\Re$ is analytic in a neighborhood of the unit hypercube $[-1,1]^m$, and that $\qsaprobe_t$ in 		\eqref{e:qSGD_probe0} satisfies (QSA0).
	Then the limit defining  $ \langle h \rangle $ in  \eqref{e:spath-notation} exists with  $h_t = h(\qsaprobe_t)$, and the following ergodic bounds hold for some constant $B_f$ independent of the phase values $\{\phi_i\}$: 
	\begin{align}
		\Bigl| \frac{1}{T} \int_0^T  \tilh(\qsaprobe_t)  \, dt  \Bigr|  &\le B_f  \frac{1}{T}\,,  \quad T>0
		\quad \text{where $\tilh = h-\langle h \rangle$.   
		}
		\label{e:BakerBoundQMC}
	\end{align} 
	Moreover,   there is a function $\hah\colon\Co^K\to \Co^K$  that is analytic in a neighborhood of the domain $\prstate\eqdef \{ z \in\Co^K :  |z_i| = 1\,, \ 1\le i\le K\}$,  real-valued on $\prstate$, and satisfying
	\begin{equation}
		\begin{aligned}
			\int_0^T  \tilh(\qsaprobe_t)  \,  dt &=   \hah(\Phi_0) -  \hah(\Phi_T)  \,,\quad T\ge 0  \,,
			\\
			\langle \hah(\Phi) \rangle &= \lim_{T\to\infty}   \frac{1}{T}    \int_0^T  \hah(\Phi_t)  \, dt   =  0  
		\end{aligned} 
		\label{e:PoissonBaker}
	\end{equation}
\end{theorem}

A version of \eqref{e:BakerBoundQMC} in    discrete time is contained in    \Cref{t:BakerCorDis}.   This result is a big surprise, given that there is so much theory predicting   $O(\log(T)^K /T)$ bounds.   Much more surprising is that $O(1/T)$ is a very poor rate of convergence for the algorithms developed in this paper, as observed in \Cref{fig:PRQSA}.   This is explained in Section~\ref{s:QMCnum}, along with examples to show how the theory can be applied to obtain convergence rates of order  $O(T^{-2+\delta})$ for specially designed QSA algorithms, and in particular new approaches to  QMC.

A few key ideas are presented here:   The assumption that $h$ is analytic is imposed so that we can first restrict to complex exponentials,  $H(\Phi_t) =   \exp(2\pi  j t \omega  t) $,  whose integral equals $ (2\pi  j  \omega  )^{-1} H(\Phi_t)$ when $\omega\neq 0$.   In the proof of  \Cref{t:BakerCor}  we use   $\omega = \sum_{i=1}^K n_i \omega_i$  with $\{n_i\}$ integers;  they are not necessarily positive, but at least one $n_i$ is assumed non-zero.  Extensions of Baker's Theorem, as surveyed in  \cite{bug18},  give a strict lower bound of the form  $|\omega | \ge \delta n^{-C}$,  where $n= 3+\sum|n_i|$ and $\delta ,C$ are non-negative constants that are independent of $n$---see \Cref{t:bug18-1.8}.   This combined with routine Taylor series bounds establishes the desired conclusions.

Bounds on the constant $B_f$ requires bounds on the constant $C$ appearing in \Cref{t:bug18-1.8}.
Current bounds on  this constant grow rapidly with $K$,  such as the doubly exponential bounds obtained in \cite{mat98} and \cite{mat00}.   Recall from \Cref{t:OneFrequency} that the nonlinearity $G_0$ in \eqref{e:qsaDynLinear} permits the creation of rich probing signals   from simple ones.  Hence,  \textit{large dimension $d$ does not mean that we   require a large value of $K$.   
}

\wham{An Elementary Ergodic Theorem}	We begin with an alternative characterization of $\barh = \langle h \rangle $.

\notes{I don't want to define this here:
	with $h_t = h(\qsaprobe_t)$}

The ergodic theorems presented here are based upon a stationary relaxation of the solution to the ODE $\ddt \Phi_t = W \Phi_t$  (recall \eqref{e:qsaDynLinear}).    We do not require (QSA0) here, but we do require that the frequencies $\{\omega_i : 1\le i\le K\}$ are \textit{distinct}.

Suppose that the initial conditions $\{\Phi_0^i : 1\le i\le K\}$ are chosen \textit{randomly} with i.i.d.\ values uniform on the unit circle $S\subset\Co$.   It follows that $\{\Phi_t^i : 1\le i\le K\}$ remain i.i.d.\  with uniform distribution for each $t\in \Re$, so that $\bfPhi =\{ \Phi_t : -\infty <t<\infty\}$ is a stationary Markov process.   Stationarity implies the Law of Large Numbers:   for any Borel measurable function $H\colon S^K\to \Co$,
\begin{equation}
	\lim_{T\to\infty}  \frac{1}{T} \int_0^T H(\Phi_t) \,  dt  =\lim_{T\to\infty}  \frac{1}{T} \int_0^T H(\Phi_{-t}) \,  dt  = \textrm{E}[ H(\Phi_0) ] =  \int_{\prstate}  H(z) \, dz \qquad a.s.
	\label{e:LLN_Phi}
\end{equation}
Conditioning on $\Phi_0 = z^0$ we can extend this limit to a.e.\ initial condition $z^0\in \prstate\eqdef S^K$.

For this stationary realization of $\bfPhi$,
we henceforth regard $\{ \qsaprobe_t = G(\Phi_t) :  t \in\Re \}$ as a  \textit{steady-state} realization  of the probing signal.

Consider any real-valued function $h$ of the probing signal \eqref{e:qSGD_probe0}.   
We have $h(\qsaprobe_t) = H(\Phi_t)$ with $H \eqdef h\circ G$ from \eqref{e:qsaDynLinear}, 
which then implies a law of large numbers.   A characterization of the limit is obtained in the following, along with a relaxation of the smoothness assumption imposed in \Cref{t:BakerCor}.

\begin{proposition}
	\label[proposition]{t:arcsinLaw}
	Suppose that
	the probing signal is defined using \eqref{e:qSGD_probe0},  with distinct  frequencies   $\{\omega_i : 1\le i\le K\}$. 
	Consider any Borel measurable function  $h\colon\Re^m\to\Re$ satisfying $ \Expect[ |h( G_0(X) ) |] <\infty$,   
	where the $K$-dimensional random vector has independent  entries,  with common distribution equal to the arcsine law on $[-\pi,\pi]$.
	
	The following limits then hold for a.e.\ set of phase angles $\{\phi_i\}$: 
	\begin{equation}
		\lim_{T\to\infty}  \frac{1}{T} \int_0^T h(\qsaprobe_t) \, dt   
		=
		\lim_{T\to\infty}  \frac{1}{T} \int_0^T h(\qsaprobe_{-t}) \, dt   
		=  \Expect[ h(G_0(X))]
		\label{e:time_reversal_stats}
	\end{equation}
	If in addition the function $h$ is continuous, then  \eqref{e:time_reversal_stats} holds for each initial condition,  and convergence is uniform in the initial phase angles.
\end{proposition}

\begin{proof}
	The result  \eqref{e:time_reversal_stats} for a.e.\ initial condition $\Phi_0$ is immediate from \eqref{e:LLN_Phi},  
	the use of $G_0$ to define the probing signal in \eqref{e:qSGD_probe0}, and the definition
	$ G(z) \eqdef G_0(  (z+ z^{-1})/2) $.
	
	To prove the stronger result for continuous $h$,  we make explicit the dependency of the average on the initial condition:
	\[
	h_t (\Phi_0)   \eqdef  h( G(e^{Wt}\Phi_0) )  =  h(\qsaprobe_t)  \,,
	\qquad  
	\barh_T (\Phi_0)   \eqdef  \frac{1}{T}\int_0^T h_t (\Phi_0)   \, dt
	\]
	We show that $\{ h_t : t>0\}$   and  $\{ \barh_t : t>0\}$  are each equicontinuous families of functions on $\prstate$: Since $e^{Wt}$ is an isometry on $\Co^K$ for any $t$, there exists a $B_v<\infty$ such that:
	\begin{equation}
		\|G(e^{Wt}z) - G(e^{Wt}z^\prime)\| < B_v\|z-z^\prime\|, \quad \text{for $z,z^\prime \in  \prstate$ and $t\ge0$}
		\label{e:isometry_of_expW}
	\end{equation}
	Now, $h$ is uniformly continuous since this  its domain is compact  (we can take its domain to be the range of $G$). 
	Consequently, for each $\epsy>0$, there exists $\delta>0$ such that for $x,x^\prime$ in the domain of $h$,
	\[
	\| h( x ) -  h( x^\prime ) \| < \delta, \quad \text{if } \|x-x^\prime\| < \epsy    
	\]
	Thus, by \eqref{e:isometry_of_expW},    
	\[
	\| h_t( z ) -  h_t( z^\prime ) \| < \delta, \quad \text{if } \|z-z^\prime\| < \epsy/B_v 
	\]
	Equicontinuity of $\{ h_t : t>0\}$   and  $\{ \barh_t : t>0\}$ on $\prstate$ follows from these bounds. 
	Pointwise convergence of  $\barh_T $ to $\langle h \rangle$ for a.e.\ $\Phi_0 \in \prstate$, as $T \to \infty$, then implies convergence from each initial condition, and also uniform convergence:
	\[
	\lim_{T \to \infty} \max_{\Phi_0 \in \prstate} \|\barh_T(\Phi_0) -  \langle h \rangle \| =0      
	\]
\end{proof}

This proposition will be refined in the following.

\smallskip

\begin{subequations}
	
	If $g\colon\Re^K\to\Re$ is analytic in a neighborhood of $z\in\Re^K$, we denote the mixed partials by
	\[
	g^{(\alpha)} (z) 
	=  \frac{\partial^{\alpha_1}}{ \partial z_1^{\alpha_1}} \cdots  \frac{\partial^{\alpha_K}}{ \partial z_K^{\alpha_K}}  g\,  (z)\,,\quad  \alpha \in \nat^K\, .
	\] 
	Denote $z^\alpha = \prod_i z_i^{\alpha_i}$,  and  $\alpha! =\prod_i \alpha_i ! $ for $ \alpha \in \nat^K $ with  $0! \eqdef 1$.    This notation is used to express the multivariate Taylor series formula in the following:  	
	\begin{lemma}
		\label[lemma]{t:TaylorSeries}
		Suppose that $g\colon \Re^K\to\Re$ 
		is analytic in a neighborhood $\clN_g$ of the hypercube $[-1,1]^K$.  
		Then there is $r_g>1$ such that whenever  $z \in[-1,1]^K$ and $0\le r \le r_g$   we have
		\begin{align}
			g(rz ) = g(0) +  & \sum_{n=1}^\infty  r^n \sum_{|\alpha|=n} \frac{1}{\alpha !} z ^\alpha   g^{(\alpha)} (0)
			\label{e:hTS}
			\\
			\text{where the sum converges absolutely:}
			\quad
			& \sum_{n=1}^\infty  r_g^n \sum_{|\alpha|=n} \frac{1}{\alpha !}  |  g^{(\alpha)} (0) |<\infty
			\label{e:absTS}
		\end{align} 
	\end{lemma}
\end{subequations}

A useful representation requires more notation:
\begin{equation}
	g_0(\qsaprobe^0_t)  \eqdef
	h(\qsaprobe_t)    = h(  G_0(\qsaprobe^0_t)) \,,\quad t\ge 0\,,
	\label{e:g0_notation}
\end{equation}
with $\qsaprobe^{0}_{t,i} = \cos(2\pi[\omega_i t + \phi_i ] ) $ as defined by \eqref{e:qSGD_probe0}.

It follows that $g_0$ is analytic if both $h$ and $G_0$ also are. Along with this new notation, \Cref{t:TaylorSeries} provides a  representation of both $h(\qsaprobe_t)$ and the integral $\tilh(\qsaprobe_t) = h(\qsaprobe_t) -\barh$. 

For $\alpha \in \nat^K$ we denote
\begin{equation}
	\begin{aligned}
		\qsaprobe_t^\alpha &= \prod_{i=1}^K   \cos(2\pi[\omega_i t + \phi_i])^{\alpha_i}
		=  2^{-|\alpha|}  \prod_{i=1}^K ( \Phi_t^i  +{  \Phi_t^i}^* )^{\alpha_i}  
		\\
		\widebar{\qsaprobe^\alpha}          & =  \langle \qsaprobe^\alpha  \rangle\,, \qquad\tilqsaprobe_t^\alpha = \qsaprobe_t^\alpha -   \widebar{\qsaprobe^\alpha}   
	\end{aligned} 
	\label{e:qsaprobe-alpha}
\end{equation}
where $|\alpha| =\sum \alpha_i$ (the $\ell_1$-norm),    and ${  \Phi_t^i}^*  $ denotes the complex conjugate.  
Let $B$ denote the set of all $K$-dimensional row vectors with entries in $\{-1,1\}$.  For fixed $\alpha \in \nat^K$,  $|\alpha|\neq 0$, we decompose $b\in B$ as follows:
$
b = (b^1, b^2,\dots, b^K)
$,
where $b^i$ has length $\alpha_i$ for each $i$,  and necessarily has entries in $\{-1,1\}$. 
These are used to define the frequency and phase variables
\begin{equation}
	\omega_{\alpha,b} = \sum_{i=1}^K   \alpha_i \omega_i \sum_k b^i_k 
	\qquad
	\phi_{\alpha,b} = \sum_{i=1}^K   \alpha_i \phi_i \sum_k b^i_k 
	\label{e:omega_ab}
\end{equation}
Let $B_0^\alpha\subset B$  denote the set of vectors $b\in B$ for which $\omega_{\alpha,b} =0$.   Under the assumption that  the frequencies  are linearly independent over the rationals, this is equivalent to the following requirement:
\[
\alpha_i \sum_k b^i_k = 0\,\quad \text{for each $1\le i\le K$}
\]

\notes{CL: added "in \eqref{e:qsaprobe-alpha}" to \Cref{t:qsaInt}
	\\
	Great!  }

\begin{lemma}
	\label[lemma]{t:qsaInt}
	The signal $\{\qsaprobe^\alpha_t\}$ in \eqref{e:qsaprobe-alpha}, its mean,  and its centered integral admit the representations,
	\begin{equation}
		\begin{aligned}
			\qsaprobe^\alpha_t &= 2^{-|\alpha|} \sum_{b\in B^\alpha}  \qsaprobe_t^{\alpha,b}
			\,,\qquad\qquad
			\widebar{\qsaprobe^\alpha}   =  2^{-|\alpha|} \sum_{b\in B_0^\alpha}  \exp(j   \phi_{\alpha,b} )
			\\
			\qsaprobe^{\alpha\intI}_t &= 2^{-|\alpha|} \sum_{b\not\in B_0^\alpha}  \frac{1}{2\pi j \omega_{\alpha,b}}  \qsaprobe_t^{\alpha,b}  \,,   \qquad
			\qsaprobe_t^{\alpha,b} =  \exp(2\pi j [ \omega_{\alpha,b}t + \phi_{\alpha,b} ]) 
		\end{aligned} 
		\label{e:probeIntDecomposed}
	\end{equation}
\end{lemma}

\begin{proof} The representation for $\qsaprobe^\alpha_t$ is purely a change of notation. 
	We have 
	$\qsaprobe_t^{\alpha,b} =   \exp(j \phi_{\alpha,b} ) $ (independent of $t$) when $b\in B_0$,   and 
	$ \langle  \qsaprobe^{\alpha,b} \rangle =0$ otherwise.   
	Consequently,
	\[
	\begin{aligned} 
		\widebar{\qsaprobe^\alpha}  
		&=  2^{-|\alpha|}   \sum_{b\in  B^\alpha}  \langle  \qsaprobe^{\alpha,b} \rangle   
		=  2^{-|\alpha|} \sum_{b\in B_0^\alpha}  \exp(j   \phi_{\alpha,b} ) 
		\\
		\qsaprobe^{\alpha\intI}_t &= \int^t_0 [\qsaprobe^\alpha_r  -  \widebar{\qsaprobe^\alpha} ] \,dr  - \bar{\qsaprobe}^{\alpha\intI}_t
		=  2^{-|\alpha|} \sum_{b\not\in B_0^\alpha}  \int^t_0  \qsaprobe_r^{\alpha,b}  \, dr - \bar{\qsaprobe}^{\alpha\intI}_t
		= 2^{-|\alpha|} \sum_{b\not\in B_0^\alpha}  \frac{1}{2 \pi j \omega_{\alpha,b}}  \qsaprobe_t^{\alpha,b}
	\end{aligned}
	\]
\end{proof}  

Before we can state the main result of this subsection we require a few more definitions.   Let  $B_+^\alpha\subset B$  denote the set of vectors $b\in B$ for which $\omega_{\alpha,b} >0$,   and  $B_-^\alpha = \{-b :  b\in  B_+^\alpha\}$.  We also require the following extension of the notation in \eqref{e:probeIntDecomposed}:
\[
z^{\alpha,b} \eqdef  z_1^{n^{\alpha,b}_1} \cdots z_K^{n^{\alpha,b}_K}\,,\quad  
\text{with } \ \ n^{\alpha,b}_i=   \alpha_i  \sum_{k=1}^{\alpha_i}  b_k^i
\,,
\quad   z\in\Co^K \setminus \{0\}
\]
where the origin is avoided because $n^{\alpha,b}_i<0$ for some $(i,\alpha, b)$.  The following properties will be useful:
\begin{equation}
	z^{\alpha,-b} = 1/ z^{\alpha,b}   \quad \text{and hence}  \quad   z^{\alpha,-b} = { z^{\alpha,b} }^* \quad \text{whenever $z\in\prstate$.}
	\label{e:z+-}
\end{equation}
where the star denotes complex conjugate.

\begin{subequations}
	
	\notes{ \rd{CL:In our submission, this statement was  "Suppose that the assumptions of \Cref{t:TaylorSeries} hold for the function $h\colon\Co^d\to\Co$ and suppose that $h$ is real-valued on $\prstate = S^K$......". I changed it for several reasons:
			1) in \Cref{t:TaylorSeries}, the domain of $g$ is $\Re^K$ and not $\Re^m$.
			2) The domain of $h$ is Real.
			Right after \eqref{e:g0_notation} I have a statement saying that analyticity of $h$ and $G_0$ imply that $g_0$ is also analytic. (QSA0) states that $G_0$ is analytic as well. I added a statement in the proof saying that This assumptions imply
			$g_0$ is analytic in a neighborhood $\clN$ of the hypercube $[-1,1]^K$ just like needed for the Taylor series of \Cref{t:TaylorSeries}.}
		\\
		This all sounds reasonable.  I think the bug might originate from our $d=K$ convention in the spring.    I need to give the proof one more careful read.}
	
	\begin{theorem}
		\label[theorem]{t:hah} 
		Suppose that the function $h\colon\Re^m\to\Re$ is analytic in a neighborhood $\clN_h$ of the hypercube $[-1,1]^m$, and suppose that $\qsaprobe_t$ is the $m$-dimensional probing signal \eqref{e:qSGD_probe0} satisfying (QSA0).
		The following conclusions then hold:
		\begin{align}
			\tilh(\qsaprobe_t) \eqdef h(\qsaprobe_t)- \barh & = \sum_{n=1}^\infty   \sum_{|\alpha|=n} \frac{1}{\alpha !}     g_0^{(\alpha)} (0)    2^{-|\alpha|} \sum_{b \in B^\alpha_+}  2  \cos(2\pi[\omega_{\alpha,b} t + \phi_{\alpha,b}])
			\label{e:tilh}
			\\
			\int_{t_0}^{t_1} \tilh(\qsaprobe_t) \, dt  & = \hah(\Phi_{t_1}) - \hah(\Phi_{t_0}) \,, \quad 0\le t_0\le t_1<\infty\,, 
			\label{e:Poisson}
			\\
			\text{where} \quad \hah(z)  &=     -\frac{1}{ 2\pi  }
			\sum_{n=1}^\infty   \sum_{|\alpha|=n} \frac{1}{\alpha !}    g_0^{(\alpha)} (0)      2^{-|\alpha|}   \sum_{b \in B_+^\alpha}  \frac{1}{   j \omega_{\alpha,b}}  [ z^{\alpha,b} 
			-  z^{\alpha,-b}  ]	 
			\label{e:hahDefined}
		\end{align}
		where $g_0:\Re^K \to \Re$ is given by \eqref{e:g0_notation}.
		
		Moreover, the function $\hah$ is analytic in the domain $\{z\in\Co^K : 0<\|z\| < r_{g_0}\}$, and admits the following representation when restricted to $\prstate$:
		\begin{equation}
			\hah(\Phi_t) =  -   \frac{1}{ \pi  } \sum_{n=1}^\infty   \sum_{|\alpha|=n} \frac{1}{\alpha !}    g_0^{(\alpha)} (0)    2^{-|\alpha|} \sum_{b \in B_+^\alpha}   \frac{1}{ \omega_{\alpha,b}}    \sin(2\pi[\omega_{\alpha,b} t + \phi_{\alpha,b}])
			\label{e:hah2}
		\end{equation}
	\end{theorem}
	
	As mentioned after \eqref{e:Poissoneqdef}, the function $\hah$ solving \eqref{e:Poisson} is known as the solution to Poisson's equation with forcing function $h$.  This terminology is standard in ergodic theory for Markov processes. Since $\hah$ is also analytic we are assured of multiple integrals  that are also bounded in time:
	\begin{equation}
		\begin{aligned}
			\hah(\Phi_T)= -\int^T_0 \tilh(\qsaprobe_t) \,  dt + \hah(\Phi_0)     \,,
			\qquad
			\hahah(\Phi_T)=  - \int^T_0 \hah(\Phi_t)  dt + \hahah(\Phi_0)
		\end{aligned} 
		\label{e:h_double_int}
	\end{equation}
	Both are required in the analysis supporting our main results:  recall the functions   \eqref{e:hatDefs} used in  \Cref{t:P-meanflow} 
	to define the terms in the p-mean flow representation.
	
\end{subequations}

The following corollary will prove useful.    Note that we relax the assumption of analyticity.

\begin{corollary}
	\label[corollary]{t:hah-g-orth}
	Suppose that (QSA0) holds,  and that  $g,h\colon \Re^m\to\Re$ are continuous functions.   Suppose moreover that there is a zero-mean solution to Poisson's equation  $\hah $,   solving \eqref{e:Poisson} for any $\Phi_0\in\prstate$, and any $ 0\le t_0\le t_1<\infty$.    Then,   $\langle g(\qsaprobe), \hah(\Phi) \rangle =0$.
\end{corollary}
The proof of the theorem and its corollary are postponed to the end of this subsection.

It is clear from \eqref{e:hahDefined}
that we require a lower bound on $| \omega_{\alpha,b} |$ for $b\in B_+^\alpha$ in order to justify that $\hah$ is analytic in a neighborhood of $\prstate$.   Useful bounds are possible through application of extensions of Baker's Theorem concerning linear independence of algebraic numbers  \cite{mat00,bug18}.

The assumption that the  $\{\omega_i\}$ defined in \eqref{e:logFreq} are  linearly independent over the field of rational numbers is equivalent to the requirement that the rational numbers $\{r_i = a_i/b_i\}$ are  \textit{multiplicatively independent}.  That is, for any integers $\{ n_i : 1\le i\le K\} \subset\intgr$, the equation  
\[
r_1^{n_1}
r_2^{n_2} \cdots
r_K^{n_K} = 1
\]
implies that $n_i=0$ for each $i$.    This is the language used in much of the literature surrounding Baker's Theorem.
The following follows from \cite[Thm. 1.8]{bug18}:
\begin{theorem}
	\label[theorem]{t:bug18-1.8}
	Under the assumptions of \Cref{t:BakerCor}  there is a constant $C>0$ depending only on   $\{a_i,b_i : 1\le i\le K\}$  
	such that whenever $\omega_{\alpha,b}  \neq 0$,
	\[
	|\omega_{\alpha,b} |   \ge \beta_\alpha^{-C}\,,  \quad \beta_\alpha = \max\{3, \alpha_1,\dots,\alpha_d\}
	\]
	\qed
\end{theorem}

\Cref{t:BakerCor}  follows immediately from \Cref{t:hah}, for which the proof is   given next:

\begin{proof}[Proof of   \Cref{t:hah}]
	The function $G_0$ is assumed analytic on $\Re^K$ under (QSA0). 
	Since $h$ is also analytic, this implies $g_0$ is analytic in a neighborhood $\clN_{g_0}$ of the hypercube $[-1,1]^K$.  
	
	The Taylor series expansion in \Cref{t:TaylorSeries} combined with \Cref{t:qsaInt}
	gives
	\begin{equation}
		\begin{aligned}
			h(\qsaprobe_t )&= g_0(0) +    \sum_{n=1}^\infty   \sum_{|\alpha|=n} \frac{1}{\alpha !}  g_0^{(\alpha)} (0)\qsaprobe_t^\alpha  
			\\
			& =     g_0(0) +\sum_{n=1}^\infty   \sum_{|\alpha|=n} \frac{1}{\alpha !}  g_0^{(\alpha)} (0) 2^{-|\alpha|} \sum_{b \in B^\alpha}   \qsaprobe_t^{\alpha,b} 
		\end{aligned}
		\label{e:hahProof1}
	\end{equation}
	Obtaining the mean of each side gives
	\[
	\barh =  g_0(0) +  \sum_{n=1}^\infty   \sum_{|\alpha|=n} \frac{1}{\alpha !}  g_0^{(\alpha)} (0) 2^{-|\alpha|} \sum_{b \in B^\alpha_0}   \qsaprobe_t^{\alpha,b} 
	\]
	where $  \qsaprobe_t^{\alpha,b}  = \exp(2\pi j \phi_{\alpha,b} )$ for $b \in B^\alpha_0$;  similar arguments were used in the derivation of 
	\eqref{e:probeIntDecomposed}.   Subtracting $\barh$ from each side of \eqref{e:hahProof1}
	gives
	\begin{equation}
		\begin{aligned}
			\tilh(\qsaprobe_t )&=       \sum_{n=1}^\infty   \sum_{|\alpha|=n} \frac{1}{\alpha !}  g_0^{(\alpha)} (0) 2^{-|\alpha|} \sum_{b \not\in B^\alpha_0}   \qsaprobe_t^{\alpha,b} 
		\end{aligned}
		\label{e:hahProof2}
	\end{equation}
	The proof of \eqref{e:tilh} is completed on observing that
	\[
	\sum_{b \not\in B^\alpha_0}   \qsaprobe_t^{\alpha,b}  
	= \sum_{b  \in B^\alpha_+}    \qsaprobe_t^{\alpha,b}  +\sum_{b  \in B^\alpha_-}    \qsaprobe_t^{\alpha,b}    
	= \sum_{b  \in B^\alpha_+}  [ \qsaprobe_t^{\alpha,b}  + \qsaprobe_t^{\alpha,-b} ]  =  2  \cos(2\pi[\omega_{\alpha,b} t + \phi_{\alpha,b}])
	\]

	The representation \eqref{e:hahProof2}  for $\tilh$  motivates the following definition:
	\begin{equation}
		\hah(\Phi_t) \eqdef  -
		\sum_{n=1}^\infty   \sum_{|\alpha|=n} \frac{1}{\alpha !}    g_0^{(\alpha)} (0)     \qsaprobe^{\alpha\intI}_t 
		\label{e:hahDef}
	\end{equation}
	whose extension to $\Co^K$ given in \eqref{e:hahDefined} is duplicated here for convenience:
	\[
	\hah(z)   =     -\frac{1}{ 2\pi  }
	\sum_{n=1}^\infty   \sum_{|\alpha|=n} \frac{1}{\alpha !}    g_0^{(\alpha)} (0)      2^{-|\alpha|}   \sum_{b \in B_+^\alpha}  \frac{1}{   j \omega_{\alpha,b}}  [ z^{\alpha,b} 
	-  z^{\alpha,-b}  ]	
	\] 
	The extension follows from the preceding arguments, using $\omega_{\alpha,-b} = -\omega_{\alpha,b} $.

	The remainder of the proof consists of two parts:   show that $\hah$ is analytic in the region $\{z\in\Co^K : 0<\|z\| < r_{g_0}\}$, and then establish the desired properties when $z\in\prstate$.   Those desired properties are firstly   $\ddt \hah(\Phi_t)  = -\tilh(\qsaprobe_t)$,  which follows from   \eqref{e:tilh}   provided the sum in  \eqref{e:hahDef}
	converges absolutely.   The final property is the representation \eqref{e:hah2}
	in terms of  sums of $ \sin(2\pi[\omega_{\alpha,b} t + \phi_{\alpha,b}])$.   This is also immediate since
	\[
	\tfrac{1}{   j }  [ z^{\alpha,b}   	-  z^{\alpha,-b}  ]	 = 2   \sin(2\pi[\omega_{\alpha,b} t + \phi_{\alpha,b}])  \quad \text{when $z= \Phi_t$}
	\]

	\smallskip
	To complete the proof we establish that $\hah$ is analytic in the given domain.
	
	From  \eqref{e:hahDefined} it follows   that $\hah$ is a function of the $2K$-dimensional vector valued function $ v(z) = (z_1,\dots,z_K, z_1^{-1}, \dots, z_K^{-1})$.   The mapping   $z\mapsto v(z)$ is analytic in $\Co^K\setminus\{0\}$, so it suffices to obtain the   bound  
	\[
	B_h(r) \eqdef 
	\sum_{n=1}^\infty   \sum_{|\alpha|=n} \frac{1}{\alpha !}   | g_0^{(\alpha)} (0)  |     2^{-|\alpha|}   \sum_{b \in B_+^\alpha}  \omega_{\alpha,b}^{-1} <\infty 
	\,,\qquad \text{ for  all  $ r < r_{g_0}$. }
	\]
	It will follow that $\hah$ is analytic in the set $\{ z\in\Co^K : 0< \|z\|< r_{g_0} \}$.
	
	\Cref{t:bug18-1.8} gives the bound  $  \omega_{\alpha,b}^{-1}  \le  3^C + n^C$  with $n=|\alpha|$.   Consequently, for any $r>0$,
	\begin{equation}
		\begin{aligned}
			B_h(r) 
			& \le   \sum_{n=1}^\infty  r^n [3^C + n^C]\sum_{|\alpha|=n} \frac{1}{\alpha !}   | g_0^{(\alpha)} (0) | 
		\end{aligned} 
		\label{e:Bh_hah}
	\end{equation}
	The right hand side is finite  for any $r< r_{g_0}$ due to the bound \eqref{e:absTS}.   
\end{proof}


We turn next to the proof of \Cref{t:hah-g-orth}, which will follow from a sequence of lemmas.

\begin{lemma}
	\label[lemma]{t:hah-g-orth-analytic}
	Suppose that the functions $g,h\colon\Re^m\to\Re$ are analytic in a neighborhood $\clN_h$ of the hypercube $[-1,1]^m$, and suppose that $\qsaprobe_t$ is the $m$-dimensional probing signal \eqref{e:qSGD_probe0} satisfying (QSA0).
	Then,   $\langle g(\qsaprobe) ,\hah(\Phi) \rangle =0$,  where $\hah$ is given in \eqref{e:hahDefined}.   
	\notes{intentionally did not use   \eqref{e:hah2} here, but reference to this eqn below is perfect.  }
\end{lemma}

\begin{proof}
	Given the representation \eqref{e:hah2} it suffices to show that   $\langle \gamma ,\zeta \rangle = 0$ for the functions  defined by $\gamma_t = \cos(2\pi[\omega_{\alpha,b} t + \phi_{\alpha,b}])$   and $\zeta_t = \sin(2\pi[\omega_{\alpha',b'} t + \phi_{\alpha',b'}])$,    with $(\alpha, b)$, $(\alpha', b')$ arbitrary pairs appearing in the sum that represents $\hah$.   
	
	If $\omega_{\alpha,b} \neq \omega_{\alpha',b'}$ then the conclusion $\langle \gamma, \zeta \rangle = 0$ is immediate (including the case $\omega_{\alpha,b} =0$).
	
	If $\omega_{\alpha,b} = \omega_{\alpha',b'}$ it follows from the definition \eqref{e:omega_ab} that $\phi_{\alpha,b} = \phi_{\alpha',b'}$, and the conclusion $\langle \gamma, \zeta \rangle = 0$ follows from the  
	double angle identity:
	\[
	\gamma_t \zeta_t   =   \sin(4\pi[\omega_{\alpha,b} t + \phi_{\alpha,b}])/2  
	\]
\end{proof}  

The next result is required for approximating $h$ and $\hah$ simultaneously by analytic functions.  Denote for $\epsy>0$,
\begin{equation}
	\hah^\epsy(z) = \int_0^\infty   e^{-\epsy t}  \tilh (e^{Wt} z)\, dt\,,\quad z\in\prstate
	\label{e:hah-epsy}
\end{equation}

\begin{lemma}
	\label[lemma]{t:hah-epsy}
	Under the assumptions of \Cref{t:hah-g-orth} we have
	\[
	\begin{aligned}
		&\hah^\epsy(z) = \hah(z)   - \epsy  \int_0^\infty   e^{-\epsy t}  \hah (e^{Wt} z)\, dt
		\\
		& \lim_{\epsy\downarrow 0} \max_{z\in\prstate} |\hah^\epsy(z) - \hah(z) | =0      
	\end{aligned} 
	\]
\end{lemma}
\begin{proof}
	The first limit follows from the differential representation of the solution to Poisson's equation:   $\hah (e^{Wt} z)$ is absolutely continuous, with  $d \hah (e^{Wt} z)  =  -\tilh  (e^{Wt} z) dt$.
	
	The second limit then follows from the assumption that the mean of $\hah$ is zero.
\end{proof}

\begin{proof}[Proof of \Cref{t:hah-g-orth}]

	\Cref{t:hah-g-orth-analytic} covers the special case in which $g, h$ are analytic.
	
	Consider next the case in which  $g$ is analytic, but $h$ and $\hah$ are only assumed continuous.    Let $\epsy>0$,   $n\ge 1$ be arbitrary,   and apply the Stone-Weierstrass Theorem to obtain a polynomial function $h_n$ satisfying $| h(x) - h_n(x)| \le \epsy/n$ for all $x\in [-1,1]^m$.   It follows from the definition \eqref{e:hah-epsy}
	that $| \hah^\epsy (z) - \hah^\epsy_n(z)| \le 1/n$ for all  $z \in\prstate$.
	
	We then have    
	\[
	\begin{aligned} 
		\langle g(\qsaprobe), \hah(\Phi) \rangle & =   \langle g(\qsaprobe), \hah^\epsy(\Phi) \rangle   + o(1)   
		\\
		&=   \langle g(\qsaprobe) ,\hah_n^\epsy(\Phi) \rangle   + o(1)  + O(\| g\|_\infty/n)   
		\\
		&=   \langle g(\qsaprobe) ,\hah_n(\Phi) \rangle   + o(1)  + O(\| g\|_\infty/n) + o_n(1)  
	\end{aligned} 
	\]
	with $ \| g\|_\infty$ the maximum of $|g(x)|$ over $x\in [-1,1]^m$ and $\hah_n$ is a polynomial function defined by \Cref{t:hah} using $h_n$. 
	In the final approximation, $o_n(1)  \to 0$ as $\epsy\downarrow 0$ for each fixed $n$, but the convergence is not necessarily uniform in $n$.   Letting $\epsy\downarrow 0$ gives,
	\[  
	\langle g(\qsaprobe) ,\hah(\Phi) \rangle  =  \langle g(\qsaprobe) ,\hah_n(\Phi) \rangle      + O(\| g\|_\infty/n)    = O(\| g\|_\infty/n)
	\]
	where the second bound holds because $g$ and $h_n$ satisfy the assumptions of \Cref{t:hah-g-orth-analytic}.    It follows that $\langle g ,\hah \rangle  = 0$ since $n\ge 1$ is arbitrary.
	
	The general case is similar but simpler:  apply the Stone-Weierstrass Theorem to obtain a polynomial function $g_n$ satisfying $| g(x) - g_n(x)| \le 1/n$ for all  $x\in [-1,1]^m$.  From the previous bound we have
	\[  
	\langle g(\qsaprobe) ,\hah(\Phi) \rangle  =  \langle g_n(\qsaprobe) ,\hah(\Phi) \rangle      + O(1/n)     = O(1/n)
	\]
	This completes the proof, since $n\ge 1$ is arbitrary.  
\end{proof}

\subsection{Extensions to Discrete Time}

 The extension of \Cref{t:BakerCor} to the discrete time setting is essentially unchanged, though obtaining bounds on the constants is more challenging \cite{wus02,bug18}.
 
 \begin{theorem}
 	\label[theorem]{t:BakerCorDis}   
 	Suppose that the assumptions of \Cref{t:BakerCor}  hold.    Then there is a finite constant $B_f$ independent of the phase values $\{\phi_i\}$ such that
 	\begin{equation}
 		\Bigl|  \frac{1}{N}   \sum_{k=1}^N      \tilh(\qsaprobe_k)        \Bigr|    \le B_f  \frac{1}{N}\,,   \quad N\ge 1 
 		\label{e:BakerBoundQMC-dis}
 	\end{equation}\qed
 \end{theorem}

 We will see in the proof that the discrete time case is more complex because we require a stronger condition on $\omega_{\alpha,b}$. 
 A useful bound is obtained   from \cite[Thm. 2.1]{bug18}:
 \begin{theorem}
 	\label[theorem]{t:bug18-2.1}
 	Under the assumptions of \Cref{t:BakerCor}  there are constants $\delta, C>0$ depending only on   $\{a_i , b_i : 1\le i\le K\}$  
 	such that whenever $\omega_{\alpha,b}  \neq 0$ and $n_0\in\intgr$,
 	\[
 	|n_0 + \omega_{\alpha,b} |   \ge \delta \beta_\alpha^{-C}\,,  \quad \beta_\alpha = \max\{3, \alpha_1,\dots,\alpha_d\}
 	\]
 	\qed
 \end{theorem}

 \begin{proof}[Proof of \Cref{t:BakerCorDis}]
 	Denote the partial sums,
 	\[
 	S_N = \sum_{k=0}^N \tilh(\qsaprobe_k)  
 	\]
 	Motivated by the foregoing, to bound the sum we consider sums of the primitives obtained from the Taylor series expansion  \eqref{e:hTS}: 
 	\[
 	\begin{aligned}
 		S_N ^{\alpha,b} & \eqdef\sum_{k=0}^N  \exp( 2\pi j [  \omega_{\alpha,b} k + \phi_{\alpha,b}]  ) 
 		\\
 		S_N ^{\alpha} & \eqdef  \sum_{b\not\in B_0^\alpha} S_N ^{\alpha,b} 
 	\end{aligned} 
 	\]
 	The expansion \eqref{e:hTS} tells us that  $S_N$ admits the representation
 	\[
 	S_N =  \sum_{n=1}^\infty   \sum_{|\alpha|=n} \frac{1}{\alpha !}    g_0^{(\alpha)} (0)   S_N ^{\alpha} 
 	\]
 	To prove the theorem, it suffices to establish a uniform bound over $N$ and $\{\phi_i\}$.
 	
 	Consider any $\alpha$ and $b\not\in B_0^\alpha$.   On denoting  $w=  \exp( 2\pi j \omega_{\alpha,b}  )$ we obtain
 	\[
 	S_N ^{\alpha,b}  =
 	\exp( 2\pi j \phi_{\alpha,b} )  \sum_{k=0}^N w^k     = \exp( 2\pi j \phi_{\alpha,b} )  \frac{w^{N+1} -1}{ w-1}
 	\,,\qquad
 	|S_N ^{\alpha,b} |   \le 2\frac{1}{| w-1| }
 	\]
 	Given $| w-1|^2 =  |1-\cos(    2\pi \omega_{\alpha,b} )|^2  + \sin( 2\pi \omega_{\alpha,b} )^2  $ it is not enough to bound 
 	$| \omega_{\alpha,b} |$ from zero as in the continuous time case.  Rather, we require a bound on $\inf_{n_0\in\intgr} |n_0+ \omega_{\alpha,b} |$.   \Cref{t:bug18-2.1} gives us the desired bound:   for some $\delta_0>0$ and any $n_0 \in \intgr$,
 	\[
 	\inf_{n_0\in\intgr} |n_0+ \omega_{\alpha,b} |   \ge\delta_0 \beta_\alpha^{-C}
 	\]
 	The remainder of the proof that $\{S_N: N\ge 1\}$ is bounded is identical to the continuous time case (see arguments surrounding \eqref{e:Bh_hah}).
 \end{proof}

	\subsection{The Perturbative Mean Flow Representation}
\label{s:p_mean_proof}
The derivation of the perturbative mean flow representation \eqref{e:BigGlobalODEvanishing} is based upon the solutions to Poisson's equation in \eqref{e:hatDefs}. We start by re-writing the QSA ODE \eqref{e:QSAgen} in terms of the \textit{apparent noise} $\{\tilXi_t\}$,
\[
\ddt \ODEstate_t = a_t [\barf(\ODEstate_t) + \tilXi_t] \, , 
\qquad 
\text{$\tilXi_t \eqdef f(\ODEstate_t,\qsaprobe_t) -  \barf(\ODEstate_t)$}
\]

The proof of the first part of \Cref{t:P-meanflow} follows directly from the next three lemmas.

\begin{lemma}
	\label[lemma]{t:step1_lemma}
	Under (QSA5), the apparent noise $\tilXi_t$ can be expressed as
	\begin{equation}
		\tilXi_t  = -\ddt  \haf(\ODEstate_t,  \Phi_t)   -a_t\Upupsilon(\ODEstate_t,  \Phi_t)
		\label{e:step1}
	\end{equation}
\end{lemma}
\begin{proof}
	Applying \eqref{e:Chainrule_diffgen}  with $h=\haf$, 
	\[
	\ddt  \haf(\ODEstate_t,  \Phi_t)   =   
	\partial_\theta \haf(\ODEstate_t,  \Phi_t)  \ddt \ODEstate_t    - 
	[ f(\ODEstate_t,  \qsaprobe_t)  -\barf(\ODEstate_t)] 
	\]
	This gives
	\[
	\tilXi_t  =   -\ddt  \haf(\ODEstate_t,  \Phi_t) 
	+    
	a_t   \partial_\theta \haf(\ODEstate_t,  \Phi_t) f(\ODEstate_t,  \qsaprobe_t)  
	\]
	Finally, \eqref{e:step1} follows from \eqref{e:Upupsilon}.
\end{proof}

\begin{lemma}
	\label[lemma]{t:step2_lemma}
	Suppose that (QSA1) and (QSA5) hold. If $a_t = (1+t)^{-\rho}$, with $\rho\in(0,1)$ , 
	\begin{equation}
		\ddt \haf(\ODEstate_t, \Phi_t )   =  -r_t a_t   [\Df \hahaf ] (\ODEstate_t, \Phi_t )   
		+ a_t \ddt  [\Df \hahaf ] (\ODEstate_t, \Phi_t ) 
		-  \tfrac{d^2}{dt^2} \hahaf(\ODEstate_t, \Phi_t ) 
		\label{e:step2}
	\end{equation}
	where $r_t = \rho/(t+1)$.
\end{lemma}
\begin{proof}
	Similarly to \Cref{t:step1_lemma}, applying \eqref{e:Chainrule_diffgen}  with $h=\hahaf$ gives
	\[
	\ddt \hahaf  (\ODEstate_t,\Phi_t)    =    
	a_t [\Df \hahaf ] (\ODEstate_t, \Phi_t )      
	-  \haf(\ODEstate_t, \Phi_t )   	
	\]
	Differentiating both sides with respect to $t$ yields
	\[
	\begin{aligned}
		\tfrac{d^2}{dt^2} \hahaf(\ODEstate_t, \Phi_t )  &= 
		\ddt\{ a_t [\Df \hahaf ] (\ODEstate_t, \Phi_t )  \} -  \ddt \haf(\ODEstate_t, \Phi_t )   
		\\
		&= -r_t a_t [\Df \hahaf ] (\ODEstate_t, \Phi_t ) + a_t \ddt [\Df \hahaf ] (\ODEstate_t, \Phi_t ) 
		-  \ddt \haf(\ODEstate_t, \Phi_t )
	\end{aligned}
	\]
	where the last equality follows from the product rule. The result in \eqref{e:step2} is then achieved upon rearranging terms.
\end{proof}

\begin{lemma}
	\label[lemma]{t:step3_lemma}
	Under (QSA5), 
	\begin{equation}
		\Upupsilon(\ODEstate_t, \Phi_t)   = \barUpupsilon(\ODEstate_t) 
		+  a_t [\Df \haUpupsilon ] (\ODEstate_t, \Phi_t ) 
		-\ddt \haUpupsilon(\ODEstate_t, \Phi_t)
		\label{e:step3}
	\end{equation}
\end{lemma}
\begin{proof}
	Again, applying \eqref{e:Chainrule_diffgen}  with $h=\haUpupsilon$,
	\[
	\ddt \haUpupsilon (\ODEstate_t,\Phi_t)    =    a_t [\Df \haUpupsilon ] (\ODEstate_t, \Phi_t )    
	- [  \Upupsilon(\ODEstate_t, \Phi_t)     -   \barUpupsilon(\ODEstate_t)] 	
	\]
	which gives \eqref{e:step3} after rearranging terms.
\end{proof}

We conclude this subsection with the remainder of the proof of  \Cref{t:P-meanflow}~(i).    Proofs of the remaining parts are given after the theorem statement:    \Cref{t:P-meanflow}~(ii) follows from  \Cref{t:hah-g-orth} and the representation  $\Upupsilon_i(\theta,\Phi) =  -\sum_{j=1}^d \haA_{i,j}( \theta, \Phi)  f_j( \theta, \qsaprobe)$,   and  (iii)  follows from \cite[Prop. 4.33 and 4.34]{CSRL}.

\begin{proof}[Proof of part (i) in \Cref{t:P-meanflow}]
	This part  contains the details of the p-mean flow representation
	\eqref{e:BigGlobalODEvanishing_thm}.   It is obtained upon substitution of \eqref{e:step2} and \eqref{e:step3} into \eqref{e:step1}.
\end{proof}

\subsection{Acceleration}
We turn next to analysis of PR and FB techniques.

\wham{Tighter bounds for PR averaging}
\label{s:Up_proofs}
As explained in \Cref{s:QSA_intro}, convergence of the QSA ODE is established by the coupling of $\{\ODEstate_t\}$ and $\{\barODEstate_t\}$ for $t \geq t_0$, for some $t_0$ that depends on the stability properties of $\dot{x} = \barf(x)$. This coupling is used to establish boundedness of the scaled error,
\begin{equation}
	Z_t \eqdef \frac{\ODEstate_t - \barODEstate_t}{a_t}, \quad t \geq t_0
	\label{e:scaled_error}
\end{equation}
Convergence of $\{\barODEstate_t\}$ to $\theta^*$ is typically very fast when $a_t=(1+t)^{-\rho}$ and $\rho<1$:
\begin{lemma}
	\label[lemma]{t:bartheta_fast} Suppose (QSA1) -- (QSA4) hold. If $a_t = (1+t)^{-\rho}$, with $0<\rho<1$, then
	\[  
	Z_t =   \frac{\ODEstate_t - \theta^*}{a_t} + \epsy^{\barODEstate}_t   
	\]
	with $\epsy^{\barODEstate}_t = [\theta^* - \barODEstate_t]/a_t$ vanishing faster than $O(T^{-N})$ for any $N \geq 1$.\qed
\end{lemma}

This lemma is part of    \cite[Prop. 4.26]{CSRL}.   Combined with  \eqref{e:Couple_ODEstate} it gives  
\begin{equation}
	Z_t = \barY^* -\haf^*_t  + O(a_t) , \quad t \geq t_0
	\label{e:QSAcouple_rho}
\end{equation}

\notes{CL: added \Cref{t:oldPR}
	\\
	OK, but it makes me sad.   The reader will wonder why the pmf isn't the foundation.  I gave it a tweak.}

The next result is  of \cite[Thm.~4.25]{CSRL}.   It is a primitive version of the p-mean flow representation that will serve as a foundation for the proof of \Cref{t:Couple_main}:
\begin{theorem}
	\label[theorem]{t:oldPR} Suppose the assumptions of \Cref{t:Couple_main} hold, then
	\begin{equation} 
		\ODEstatePR_T = \theta^* + a_T c(\rho,\kappa) \barY^* + \mathcal{B}_T/T
		\label{e:PR_1/T}
	\end{equation} 
	where $\{\mathcal{B}_T\}$ is bounded in $T$ and $c(\rho,\kappa)$ is defined in \Cref{t:Couple_main} .\qed
\end{theorem}

The key step in the proof of \Cref{t:Couple_main} is to bound the process $\{\mathcal{B}_T\}$ in \eqref{e:PR_1/T}. It is expressed in \cite[Thm. 4.25]{CSRL} as 
\begin{equation}
	\mathcal{B}_T = [A^\ocp]^{-1} \{ \epsy^{Y}_T - \epsy^{Z}_T + \epsy^{\Upupsilon}_T + \epsy^{a}_T \} 
	\label{e:BddProcessPR}
\end{equation}
where for $Y_t \eqdef Z_t + \haf(\ODEstate_t,\Phi_t)$ (recall \eqref{e:PoissonA1}), $r_t=\rho/(1+t)$ and $\tilUpupsilon^*_t = \Upupsilon^*_t - \barUpupsilon^*$ (recall \eqref{e:hatDefs} and \eqref{e:Upupsilon}),
\begin{equation*}
	\epsy^{Y}_T = Y_T - Y_{T0},
	\qquad  
	\epsy^{Z}_T =  \int_{T_0}^T r_t Z_t \,dt, 
	\qquad
	\epsy^{\Upupsilon}_T = \int_{T_0}^Ta_t \tilUpupsilon^*_t \,dt,
	\qquad
	\epsy^{a}_T = \int^T_{T_0} a_t O(\|\ODEstate_t - \theta^*\| )\,dt
\end{equation*}

We proceed to bound each term.

\begin{lemma}
	\label[lemma]{t:epsyY_bound} Under the assumptions of \Cref{t:Couple_main}, $\epsy^{Y}_T  =O(a_T)$.
	\begin{proof}
		By substitution of \eqref{e:QSAcouple_rho} into the definition of $Y_t$,
		\begin{align*}
			Y_t &= \barY^* - \haf^*_t + \haf(\ODEstate_t, \Phi_t) + O(a_t)  \\
			& =\barY^* + O(\| \ODEstate_t - \theta^*\|) +  O(a_t) =  \barY^*+ O(a_t) 
		\end{align*}
		where the last equalities follow from Lipschitz continuity of $\haf_t(\ODEstate_t, \Phi_t)$ together with the fact that $\|\ODEstate_t - \theta^*\| = O(a_t)$ from \Cref{t:Couple_main}. Now, $a_{T_0} = O(a_T)$ since $T_0 = (1-1/\kappa)T$ and hence
		\[ 
		\epsy^{Y}_T = Y_T - Y_{T0} = O(a_T) 
		\]
	\end{proof}
\end{lemma}

\begin{lemma}
	\label[lemma]{t:epsyZ_bound} Under the assumptions of \Cref{t:Couple_main},
	\[  
	\epsy^{Z}_T = \epsy^{Z}_\infty + O(a_T), \qquad \epsy^{Z}_\infty =  \barY^*\log(\kappa/(\kappa-1))  \rho 
	\]
\end{lemma}

\begin{proof}
	For $r_t$ as defined below \eqref{e:BddProcessPR} and $a_t,T_0$ as defined by \Cref{t:Couple_main}, we have the following:
	\[ 
	\int_{T_0}^T  r_t \,dt
	= \rho\log(\kappa/(\kappa-1)), \qquad
	\int_{T_0}^T  r_t O(a_t)\,dt
	= O(a_T)
	\]
	The above identities along with the representation of $Z_t$ in \eqref{e:QSAcouple_rho} imply
	\begin{equation}
		\epsy^{Z}_T =  \int_{T_0}^T  r_t[  \barY^* - \haf^*_t + O(a_t) ] \,dt  = \barY^*\log(\kappa/(\kappa-1))  \rho    + \int_{T_0}^T  r_t \haf^*_t  \,dt + O(a_T)
		\label{e:epsyz}
	\end{equation}
	It remains to bound the last integral in the right side of \eqref{e:epsyz}. By integration by parts,
	\begin{equation*}
		\int_{T_0}^T  r_t \haf^*_t \,dt  =     \int_{T_0}^T  r_t  \, d\hahaf^*_t   = r_T \hahaf^*_T      - r_{T_0} \hahaf^*_{T_0} +  \frac{1}{\rho} \int_{T_0}^T  r^2_t  \hahaf^*_t \,dt
	\end{equation*}
	Here, $\hahaf^*_t$ is bounded by assumption in (QSA5) so we have the bound
	\[
	\int_{T_0}^T r^2_t  \hahaf^*_t \,dt \leq \sup_t\| \hahaf^*_t \| \int_{T_0}^T r^2_t \, dt = O(r_T)
	\]
	Thus, $\int_{T_0}^T  r_t \haf^*_t  \,dt = O(r_T)$. The result then follows from substitution of this bound into \eqref{e:epsyz}.
\end{proof}

\begin{lemma}
	\label[lemma]{t:epsyUp_bound} Under the assumptions of \Cref{t:Couple_main}, $ \| \epsy^{\Upupsilon}_T \|= O(a_T)$.
	
	\begin{proof}
		Again applying integration by parts,
		\begin{equation}
			\epsy^{\Upupsilon}_T = - \int_{T_0}^Ta_t \,d\haUpupsilon^*_t  = a_T\haUpupsilon^*_T - a_{T_0}\haUpupsilon^*_{T_0} + \int_{T_0}^T  r_t a_t \haUpupsilon^*_t \,dt
			\label{e:epsyup}
		\end{equation}
		where $\haUpupsilon^*_t$ is as defined by \eqref{e:haUpupsilon}. We have that $\haUpupsilon^*_t$ is bounded by assumption in (QSA5) and we have the bound 
		\[
		\int_{T_0}^T r_t a_t   \haUpupsilon^*_t \,dt \leq \sup_t\| \haUpupsilon^*_t \| \int_{T_0}^T r_t a_t \,dt =  O(a_T)
		\]
		Then, we obtain the desired conclusion by substitution of the above bound into \eqref{e:epsyup}.	
	\end{proof}
\end{lemma}

\begin{lemma}
	\label[lemma]{t:epsya_bound} Under the assumptions of \Cref{t:Couple_main}, $\epsy^{a}_T = O(T^{1-2\rho})$.
	\begin{proof}
		From \Cref{t:Couple_main}, $\|\ODEstate_t - \theta^*\| = O(a_t)$. Thus,
		\begin{equation*}
			\epsy^{a}_T = 	\int^T_{T_0} a_t O(\|\ODEstate_t - \theta^*\| ) \, dt = \int^T_{T_0} O(a^2_t) \,dt = O(T^{1-2\rho})
		\end{equation*}	
	\end{proof}
\end{lemma}
\smallskip
\begin{proof}[Proof of \Cref{t:Couple_main}]
	The proof of \eqref{e:Couple_ODEstate} can be found in \cite[Thm. 4.24]{CSRL}.
	
	Combining \Cref{t:epsyY_bound,t:epsyZ_bound,t:epsyUp_bound,t:epsya_bound}:
	\[   
	\mathcal{B}_T   =  -[A^\ocp]^{-1} \barY^*\log(\kappa/(\kappa-1))  \rho  +O(T^{1-2\rho})   
	\]	
	Then, \eqref{e:PRtheta} follows from substitution of the above representation into \eqref{e:PR_1/T}.	
\end{proof}

\wham{Forward-backward filtering}
We refer to $\qsaprobe_t$ as the forward probing signal, and $\qsaprobe^-_t$ as backward probing.
Similarly to $\qsaprobe^-_t$, we introduce the signal $\Phi^-_t \eqdef \Phi_{-t}$.
The first step in establishing the identity $\barY^{*-} = - \barY^*$ is to show that the solutions to Poisson's equation \eqref{e:haA_def} for the two QSA ODEs differ by a negative sign.

\begin{lemma}		
	\label[lemma]{t:Back_Poisson}
	Suppose that  (QSA5) holds, and let $\haA^{*}$ and $\haA^{*-}$ denote the solutions to Poisson's equation satisfying:
	\[
	\begin{aligned}
		\haA^{*}(  \Phi_{t_0} )  & = 
		\int_{t_0} ^{t_1}   \tilA( \theta^\ocp ,  \qsaprobe_t)   \, dt    + \haA^{*}(  \Phi_{t_1} )  
		\\
		\haA^{*-}(  \Phi_{t_0}^-)  & = 
		\int_{t_0} ^{t_1}   \tilA( \theta^\ocp ,  \qsaprobe_t^-)   \, dt    + \haA^{*-}(  \Phi_{t_1}^- ) \, ,   \quad \text{any  $0\le t_0\le t_1$,}
	\end{aligned} 
	\]
	normalized so that 		 $\langle \haA^{*} \rangle =\langle \haA^{*-} \rangle =0$,  as in \eqref{e:Ameanzero}.
	We then have  $\haA^{*}( z) = -\haA^{*-}( z) $ for $z\in\prstate$.   
\end{lemma}

\begin{proof} Through a change of variables when $t_0 =0$, $t_1=t$,
	\[
	\haA^-(\Phi^-_t)  =  \int^{-t}_{0} - \tilA( \theta^\ocp ,  \qsaprobe_{\tau})  \,d\tau + \haA^{*-}(\Phi^-_0) 
	\]
	Differentiating each side gives
	\begin{equation*}
		\ddt \haA^{*-}(\Phi^-_t)  =  \tilA( \theta^\ocp ,  \qsaprobe_t^-) = - \ddt 	\haA^{*}(\Phi^-_t)
	\end{equation*}
	That is, $\haA^{*}(z) = -\haA^{*-}(z)+M^\circ$ for a constant matrix $M^\circ \in \Re^{d \times d}$. 
	The conclusion $M^\circ =0$ follows from the assumed normalization on the means of $\haA^{*}$ and $\haA^{*-}$.
\end{proof}

%

Recall from \eqref{e:barY} that $\barUpupsilon^*$ and its analogous quantity $\barUpupsilon^{*-}$ for the QSA ODE with backward probing \eqref{e:Backwards_QSA} can be expressed
\begin{subequations}
	\begin{align}
		\barUpupsilon^*&=- \lim_{T \to \infty} \frac{1}{T} \int^T_{0} \haA^{*}(\Phi_t) f(\theta^\ocp,\qsaprobe_t)\,dt
		\label{e:Upsilon_pos}\\
		\barUpupsilon^{*-}&=- \lim_{T \to \infty} \frac{1}{T} \int^T_{0} \haA^{*-}(\Phi^-_t) f(\theta^\ocp,\qsaprobe^-_t)\,dt
		\label{e:Upsilon_neg}
	\end{align}
\end{subequations}
where $\haA^*$ and $\haA^{*-}$ are as defined by \Cref{t:Back_Poisson}. We now show that the application of \Cref{t:Back_Poisson} and \Cref{t:arcsinLaw} to \eqref{e:Upsilon_pos} and \eqref{e:Upsilon_neg} leads to the proof of \Cref{t:FB_ROC}.
\smallskip
\begin{proof}[Proof of \Cref{t:FB_ROC}]
	From \Cref{t:Couple_main}, the following holds for $\rho \in (1/2,1)$:
	\begin{align}
		\ODEstateFB_T &= \half  [\ODEstatePR_T+\ODEstatePRm_T] \nonumber\\
		&= \theta^\ocp+  \half  {a_T c(\kappa,\rho)} (\barY^*+\barY^{*-})+ \half  (\mathcal{B}_T + \mathcal{B}^-_T)/T
		\label{e:FB_eq}
	\end{align}
	An application of the backwards in time LLN in \Cref{t:arcsinLaw} with $h(\qsaprobe_t) =-\haA^{*}(\Phi_t) f(\theta^\ocp,\qsaprobe_t)$ along with \eqref{e:Upsilon_pos} yields
	\begin{equation*}
		\barUpupsilon^* =- \lim_{T \to \infty} \frac{1}{T} \int^T_{0} \haA^{*}(\Phi^-_t) f(\theta^\ocp,\qsaprobe^-_t)\,dt
	\end{equation*}
	Then, by \Cref{t:Back_Poisson}, we get that $\haA^{*}(\Phi^-_t) f(\theta^\ocp,\qsaprobe^-_t) = -\haA^{*-}(\Phi^-_t) f(\theta^\ocp,\qsaprobe^-_t)$. Together with \eqref{e:Upsilon_neg} we obtain
	\begin{equation*}
		\barUpupsilon^{*-} = -\lim_{T \to \infty} \frac{1}{T} \int^T_{0} \haA^{*-}(\Phi^-_t) f(\theta^\ocp,\qsaprobe^-_t)\,dt=  \lim_{T \to \infty} \frac{1}{T} \int^T_{0} \haA^{*}(\Phi^-_t) f(\theta^\ocp,\qsaprobe^-_t)\,dt = - \barUpupsilon^* 
	\end{equation*}
	which implies $\barY^{*-} =-\barY^*$ via \eqref{e:barY}. Combining this result with \eqref{e:FB_eq} completes the proof.
\end{proof}
	
	\subsection{QSA applied to Gradient-Free Optimization}
\label{s:GFO_proofs}
Recall that $\Obj: \Re^d \to \Re$ is strongly convex if there exists a constant $\delta_0>0$ such that:
\begin{equation}
	\delta_0 \| \theta -\theta^0\| ^2\leq 	(\nabla\Obj(\theta) -\nabla\Obj(\theta^0))\cdot (\theta -\theta^0) \, , \quad \forall \,\theta, \theta^0 \in \Re^d
	\label{e:strong_convex_def}
\end{equation}

The bias bound in \Cref{s:GFO_intro} is a corollary of the fact that $1$qSGD and $2$qSGD have identical average vector fields under mild assumptions on $\bfqsaprobe$. This follows from \Cref{t:3sin} and \Cref{t:samefbar}.

\begin{lemma}
	\label[lemma]{t:3sin}
	Suppose $\bfqsaprobe$ is defined using \eqref{e:qSGD_probe},  with $\{\omega_i\}$ distinct. Then,
	for any constant matrix $M \in \Re^{d \times d}$,   
	\[
	\langle\qsaprobe \qsaprobe^\transpose M \qsaprobe\rangle = 0
	\]
\end{lemma}

\begin{proof}
	We apply \Cref{t:arcsinLaw}.   The arcsine law used in the proposition is odd, meaning that $\qsaprobe $ and $-\qsaprobe$ have the same steady-state distribution, giving 
	$\langle\qsaprobe \qsaprobe^\transpose M \qsaprobe\rangle 
	= -\langle\qsaprobe \qsaprobe^\transpose M \qsaprobe\rangle =0 
	$.
\end{proof}

\smallskip

\begin{proof}[Proof of \Cref{t:samefbar}] 
	Following the proof of \Cref{t:3sin}, $\qsaprobe$ and $-\qsaprobe$ have the same steady-state distribution by \Cref{t:arcsinLaw}, which gives
	\[  
	\langle\fq(\theta,\qsaprobe)\rangle  =  -\frac{1}{\epsy}\langle \qsaprobe \Obj(\theta+\epsy\qsaprobe)\rangle 
	=  \frac{1}{\epsy} \langle \qsaprobe \Obj(\theta-\epsy\qsaprobe) \rangle 
	= \langle\fq(\theta,-\qsaprobe)\rangle 
	\] 
	This implies equality of the average vector fields for $\fq$ and $\fqq$:
	\begin{align*}  
		\barf(\theta)  = \langle\fqq(\theta,\qsaprobe)\rangle &=  -\frac{1}{2\epsy} \langle\qsaprobe \Obj(\theta+\epsy\qsaprobe) \rangle + \frac{1}{2\epsy} \langle \qsaprobe \Obj(\theta-\epsy\qsaprobe)\rangle \\
		&= -\frac{2}{2\epsy}\langle \qsaprobe \Obj(\theta+\epsy\qsaprobe) \rangle  =  \langle\fq(\theta,\qsaprobe) \rangle
	\end{align*}
	Now, by a second order Taylor series expansion of $\fq(\theta,\qsaprobe_t)$ around $\theta$,
	\[
	\fq(\theta,\qsaprobe_t) = -\frac{1}{\epsy}\qsaprobe_t\Obj(\theta)- \qsaprobe_t \qsaprobe_t^\transpose \nabla\Obj(\theta)- \frac{\epsy}{2} \qsaprobe_t \qsaprobe_t^\transpose \nabla^2\Obj(\theta)\qsaprobe_t + O(\epsy^2)
	\]
	Taking the mean of each side yields
	\[
	\barf(\theta) = - \Sigmaqsa \nabla\Obj(\theta)- \frac{\epsy}{2} \langle \qsaprobe \qsaprobe^\transpose \nabla^2\Obj(\theta)\qsaprobe \rangle + O(\epsy^2)
	\]
	We have $ \langle \qsaprobe \qsaprobe^\transpose \nabla^2\Obj(\theta)\qsaprobe \rangle =0$ by \Cref{t:3sin}, which gives \eqref{e:barf_qsgd}.
\end{proof}

\smallskip
The proof of \Cref{t:bias_qSGD1_strconvx} then follows from applying the results of \Cref{t:samefbar} to \eqref{e:strong_convex_def}.

\begin{proof}[Proof of \Cref{t:bias_qSGD1_strconvx}]
	Since $\Obj$ satisfies \eqref{e:strong_convex_def} and $\nabla\Obj(\thetaopt)=0$, we achieve the following for $\delta>0$:
	\begin{align*}
		\delta \| \theta^{\ocp} -\thetaopt\| ^2 &\leq	(\nabla\Obj(\theta^{\ocp} ) -\nabla\Obj(\thetaopt))\cdot (\theta^{\ocp}  -\thetaopt) \\
		\delta \| \theta^{\ocp}  -\thetaopt\| ^2 & \leq \| \nabla\Obj(\theta^{\ocp})\| \, \| \theta^{\ocp} -\thetaopt\| 
	\end{align*} 
	We have that $\barf(\theta^*) =0$ by (QSA3), which implies $\nabla\Obj(\theta^{\ocp}) = O(\epsy^2) $ from \eqref{e:barf_qsgd} under the assumption that $\Sigmaqsa > 0$. Thus,
	\[
	\| \theta^{\ocp}  -\thetaopt\|   \leq   O(\epsy^2)
	\]
\end{proof}

	\end{document}